\documentclass{mcom-l}
\usepackage{xcolor}
\colorlet{siaminlinkcolor}{green!50!black}
\colorlet{siamexlinkcolor}{red!100!black}
\usepackage{amssymb}
\usepackage{bbm}

\usepackage{graphicx,float}

\makeatletter
\@namedef{subjclassname@2020}{\textup{2020} Mathematics Subject Classification}
\makeatother

\makeatletter
\newcommand{\thickhline}{%
	\noalign {\ifnum 0=`}\fi \hrule height 1.2pt
	\futurelet \reserved@a \@xhline
}
\makeatother

\newcommand{\R}{\mathbb{R}}
\newcommand{\C}{\mathbb{C}}

\newcommand{\vx}{\mathbf{x}}
\newcommand{\dxp}{{\delta_x^+}}

\newcommand{\dxx}{{\delta_x^2}}

\newcommand{\rmd}{\mathrm{d}}

\newcommand{\vep}{\varepsilon}

\newcommand{\pbtt}[1]{{\Phi_{B_2}^t(#1)}}

\newcommand{\B}[1]{B(#1)}
\newcommand{\fe}{f_\vep}

\newcommand{\G}[1]{G(#1)}
\newcommand{\expit}[1]{e^{- i \tau (V + f(|#1|^2))}}
\newcommand{\expitVjtheta}[1]{e^{- i \tau (V_j^\theta + f(|#1|^2))}}
\newcommand{\vini}{v_0}
\newcommand{\p}{^\prime}
\newcommand{\pp}{^{\prime \prime}}

\usepackage{amsopn}
\DeclareMathOperator{\sinc}{sinc}

\usepackage[hidelinks]{hyperref}
\usepackage[capitalize,nameinlink,noabbrev]{cleveref}

\newtheorem{theorem}{Theorem}[section]
\newtheorem{lemma}[theorem]{Lemma}
\newtheorem{proposition}[theorem]{Proposition}
\newtheorem{corollary}[theorem]{Corollary}

\theoremstyle{definition}

\theoremstyle{remark}
\newtheorem{remark}[theorem]{Remark}

\numberwithin{equation}{section}
\numberwithin{figure}{section}

\crefname{example}{Example}{Examples}
\crefname{hypothesis}{Hypothesis}{Hypotheses}
\crefname{conj}{Conjecture}{Conjectures}
\crefformat{equation}{\textup{#2(#1)#3}}
\crefrangeformat{equation}{\textup{#3(#1)#4--#5(#2)#6}}
\crefmultiformat{equation}{\textup{#2(#1)#3}}{ and \textup{#2(#1)#3}}
{, \textup{#2(#1)#3}}{, and \textup{#2(#1)#3}}
\crefrangemultiformat{equation}{\textup{#3(#1)#4--#5(#2)#6}}%
{ and \textup{#3(#1)#4--#5(#2)#6}}{, \textup{#3(#1)#4--#5(#2)#6}}{, and \textup{#3(#1)#4--#5(#2)#6}}

\Crefformat{equation}{#2Equation~\textup{(#1)}#3}
\Crefrangeformat{equation}{Equations~\textup{#3(#1)#4--#5(#2)#6}}
\Crefmultiformat{equation}{Equations~\textup{#2(#1)#3}}{ and \textup{#2(#1)#3}}
{, \textup{#2(#1)#3}}{, and \textup{#2(#1)#3}}
\Crefrangemultiformat{equation}{Equations~\textup{#3(#1)#4--#5(#2)#6}}%
{ and \textup{#3(#1)#4--#5(#2)#6}}{, \textup{#3(#1)#4--#5(#2)#6}}{, and \textup{#3(#1)#4--#5(#2)#6}}

\hypersetup{
	allcolors = siaminlinkcolor,
	linkcolor = siamexlinkcolor,
}

\begin{document}

\title[Error estimates for NLSE with semi-smooth nonlinearity]{Error estimates of the time-splitting methods for the nonlinear Schr\"odinger equation with semi-smooth nonlinearity}


\author[W. Bao]{Weizhu Bao}
\address{Department of Mathematics, National University of Singapore, Singapore 119076}
\email{matbaowz@nus.edu.sg}
\thanks{This work was partially supported by the Ministry of Education of Singapore under its  AcRF Tier 2 funding MOE-T2EP20122-0002 (A-8000962-00-00) (W. Bao). }

\author[C. Wang]{Chushan Wang}
\address{Department of Mathematics, National University of Singapore, Singapore 119076}
\email{E0546091@u.nus.edu}

\subjclass[2020]{Primary 35Q55, 65M15, 65M70, 81Q05}

\date{}
\keywords{nonlinear Schr\"odinger equation, semi-smooth nonlinearity, time-splitting pseudospectral method, error estimate, local regularization}

\begin{abstract}
We establish error bounds of the Lie-Trotter time-splitting sine pseudospectral method for the nonlinear Schr\"odinger equation (NLSE) with semi-smooth nonlinearity $ f(\rho) = \rho^\sigma$, where $\rho=|\psi|^2$ is the density with $\psi$ the wave function and $\sigma>0$ is the exponent of the semi-smooth nonlinearity. Under the assumption of $ H^2 $-solution of the NLSE,
we prove error bounds at $ O(\tau^{\frac{1}{2}+\sigma} + h^{1+2\sigma}) $ and $ O(\tau + h^{2}) $ in $ L^2 $-norm for $0<\sigma\leq\frac{1}{2}$ and
$\sigma\geq\frac{1}{2}$, respectively, and  an error bound at
$ O(\tau^\frac{1}{2} + h) $ in $ H^1 $-norm for $\sigma\geq \frac{1}{2}$, where
$h$ and $\tau$ are the mesh size and time step size, respectively.
In addition,
when $\frac{1}{2}<\sigma<1$ and under the assumption of $ H^3 $-solution of the NLSE, we show an error bound  at $ O(\tau^{\sigma} + h^{2\sigma}) $ in $ H^1 $-norm. Two key ingredients are adopted
in our proof: one is to adopt an unconditional $ L^2 $-stability of the numerical flow in order to avoid an a priori estimate of the numerical solution
for the case of $ 0 < \sigma \leq \frac{1}{2}$, and to establish an $ l^\infty $-conditional $ H^1 $-stability to obtain the $ l^\infty $-bound of the numerical solution by using
the mathematical induction and the error estimates for the case of $ \sigma \ge \frac{1}{2}$; and the other one is to introduce a regularization technique to avoid the singularity of the semi-smooth nonlinearity in obtaining improved local truncation errors. Finally, numerical results are reported to demonstrate our error bounds.
\end{abstract}

\maketitle


\section{Introduction}
In this paper, we consider the following nonlinear Schr\"odinger equation (NLSE)
\begin{equation}\label{NLSE}
	i \partial_t \psi(\vx, t) = -\Delta \psi(\vx, t) + V(\vx) \psi(\vx, t) + f(|\psi(\vx, t)|^2) \psi(\vx, t), \quad \vx \in \Omega, \quad t>0,
\end{equation}
with the initial data
\begin{equation}\label{init}
	\psi(\vx, 0) = \psi_0(\vx), \quad \vx \in \overline{\Omega},
\end{equation}
and the homogeneous Dirichlet boundary condition
\begin{equation}\label{bound}
\psi(\vx, t)=0, \quad \vx \in \partial \Omega, \quad t \geq 0, 
\end{equation}
where $t$ is time, $\vx\in \R^d$ ($d=1, 2, 3$) is the spatial coordinate, $ \psi:=\psi(\vx, t) $ is a complex-valued wave function, and $ V:=V(\vx): \Omega \rightarrow \R $ is a time-independent real-valued potential. Here $ \Omega = \Pi_{i=1}^d (a_i, b_i) \subset \R^d $ is a bounded domain, and the nonlinearity is given as
\begin{equation}\label{semi-smooth}
	f(\rho) = \beta \rho^\sigma, \quad \rho:=|\psi|^2 \geq 0,
\end{equation}
where $\beta \in \R$ is a given constant and $\sigma > 0$ is the exponent of the nonlinearity.
The NLSE \cref{NLSE} conserves the mass
\begin{equation}
	M(\psi(\cdot, t)) = \int_{\Omega} |\psi(\vx, t)|^2 \rmd \vx \equiv M(\psi_0), \quad t \geq 0,
\end{equation}
and the energy
\begin{equation}
\begin{aligned}
E(\psi(\cdot, t))&= \int_{\Omega} \left[ |\nabla \psi(\vx, t)|^2 + V(\vx)|\psi(\vx, t)|^2 + F(|\psi(\vx, t)|^2) \right] \rmd \vx \\
&\equiv E(\psi_0), \quad t \geq 0,
	\end{aligned}
\end{equation}
where the interaction energy density $F(\rho)$ is given as
\begin{equation}\label{eq:F}
	F(\rho) = \int_0^\rho f(s) \rmd s = \frac{\beta}{\sigma+1} \rho^{\sigma+1}, \quad \rho \geq 0.
\end{equation}

When $\sigma=1$ in \eqref{semi-smooth}, i.e. $ f(\rho) = \beta \rho $ and
$ F(\rho) = \frac{\beta}{2} \rho^2 $, \cref{NLSE} collapses to the well-known nonlinear Schr\"odinger equation with cubic nonlinearity (or smooth nonlinearity), also known as the Gross-Pitaevskii equation (GPE), which has been widely adopted for modeling and simulation in quantum mechanics, nonlinear optics, and Bose-Einstein condensation \cite{review_2013,ESY,NLS}. Arising from different physics applications,
semi-smooth nonlinearity is introduced in the NLSE \cref{NLSE}, i.e.
$\sigma$ is taken as a non-integer in \eqref{semi-smooth}.
Typical examples include, in the Schr\"{o}dinger-Poisson-X$ \alpha $ model
with $ f(\rho) = -\alpha\rho^{1/d} (\alpha>0) $ \cite{bao2003,SPXalpha},
i.e. $\sigma=\frac{1}{3}$ and $\sigma=\frac{1}{2}$ in three dimensions (3D) and
two dimensions (2D), respectively; in the LHY correction (a next-order correction of the ground state energy proposed by Lee, Huang and Yang in 1957 \cite{LHY}) for a beyond-mean-field term which is widely adopted in modeling and simulation for quantum droplets \cite{QD1,QD2,QD3,QD4,QD5} with $ f(\rho) = \rho^{3/2} $ in 3D, i.e. $\sigma=\frac{3}{2}$, $ f(\rho) = \sqrt{\rho} $ in one dimension (1D), i.e. $\sigma=\frac{1}{2}$, and $ f(\rho) = \rho \ln \rho $ in 2D; and in the mean field
model for Bose-Fermi mixture \cite{Had,Cai}, with $f(\rho)=\rho^{2/3}$, i.e. $\sigma=\frac{2}{3}$. For all the aforementioned nonlinearities ({actually for all $ \sigma > 0 $ when $d=1,2,3$}), the NLSE \cref{NLSE} is well-posed in $ H^2 $ {under suitable assumptions on $V$, e.g. $ V \in L^p $ with $p \geq 1$ and $p > d/2$ \cite{kato1987,cazenave2003}.} However, to our best knowledge, there is no guarantee of higher regularity to be propagated due to the low regularity of the semi-smooth nonlinearity, which is similar to the case of the logarithmic Schr\"odinger equation (LogSE) \cite{sinum2019,bao2019,bao2022}. In fact, similar to the LogSE, the low regularity of the solution of the NLSE with semi-smooth nonlinearity is mainly due to the low regularity of the nonlinearity. {We remark here that the potential $V$ could also be a source of low regularity of the solution, however, we will not consider the low regularity of $V$ in this paper but leave it as our future work. }

For the cubic NLSE, i.e. $ \sigma = 1 $, many accurate and efficient numerical methods have been proposed and analyzed in the last two decades, including the finite difference method \cite{FD,bao2013,review_2013,Ant}, the exponential wave integrator \cite{bao2014,ExpInt,SymEWI}, the time-splitting method \cite{bao2003JCP,BBD,lubich2008,schratz2016,review_2013,splitting_low_reg,Ant}, the finite element method \cite{FEM1,FEM2,FEM3,FEM4,henning2017}, etc. Recently, new low regularity integrators or resonance based Fourier integrators are designed and analyzed for the cubic NLSE with low regularity initial data since the important work by Ostermann and Schratz \cite{LRI}, followed by \cite{LRI_sinum,LRI_error,LRI_general,LRI_fulldisc,LRI_sec} and references therein for different dispersive partial differential equations. For all these numerical methods, 
optimal error bounds were rigorously established under different regularity assumptions
of the cubic NLSE. 

Most numerical methods for the cubic NLSE can be extended straightforwardly to solve the NLSE \cref{NLSE} with non-integer $ \sigma >0$, e.g. semi-smooth nonlinearity with $0<\sigma<1$, which is different from the NLSE with singular nonlinearity, where regularization may be needed \cite{sinum2019,bao2019,bao2022,bao2023singular}. However, due to the low regularity of solution of the NLSE \cref{NLSE}
with semi-smooth nonlinearity and the low regularity of the semi-smooth nonlinearity \eqref{semi-smooth} in the NLSE \cref{NLSE} which causes
order reduction in local truncation errors and results in difficulties in obtaining stability estimates, error analysis for different numerical methods applied to \cref{NLSE} with non-integer $ \sigma >0$ is a very subtle and challenging 
question! For example, first order temporal convergence of the finite difference method requires boundedness of the second-order time derivative, which roughly requires the exact solution to be in $ H^4 $, which is beyond the regularity property of the NLSE 
\cref{NLSE} with semi-smooth nonlinearity. In fact, based on our numerical experiments with a smooth initial datum $ \psi_0(x) = xe^{-x^2/2} $, it indicates that $ \psi(\cdot, t) \not \in H^4 $ for $ t>0 $ and $ \sigma $ small! Since the time-splitting methods usually need lower regularity requirements on the exact solution than the finite difference methods, in this work, we consider the time-splitting method and in particular the first-order Lie-Trotter splitting method due to the low regularity of the semi-smooth nonlinearity and the low regularity of the exact solution of \cref{NLSE}.

Error estimates of the time-splitting methods with different orders for the cubic NLSE (i.e. $ \sigma =1 $) have been well understood and we refer the readers to \cite{lubich2008,schratz2016,splitting_low_reg,Ant} and references therein.
However, for the NLSE with non-integer $ \sigma $, only limited results are established for the filtered Lie-Trotter splitting scheme which requires a strong CFL-type time step size restriction $\tau=O(h^2)$. In \cite{ignat2011}, first order convergence in $ L^2 $-norm is established for $ H^2 $-solution and $ \sigma \geq 1/2 $. Then generalized in \cite{choi2021}, half order convergence in $ L^2 $-norm is established for $ H^1 $-solution and $ \sigma>0 $. These convergence rates are optimal with respect to the regularity assumptions on the exact solution. However, there are still some questions related to error estimates to be addressed: (i) it is unclear whether higher convergence order can be obtained for $ H^2 $-solution when $ 0 < \sigma < 1/2 $; (ii) their results are established for the filtered Lie-Trotter scheme, which is a semi-discretization scheme with a specific strong CFL-type time step size restriction, and it loses mass conservation and time symmetric property in the discretized level; and (iii) there is no optimal error estimate in $ H^1 $-norm, which is the natural norm of the NLSE. 

The main aim of this paper is to establish error estimates of the time-splitting sine pseudospectral (TSSP) method \cref{full_discretization_scheme} for the NLSE \cref{NLSE} with semi-smooth nonlinearity. We remark here that the TSSP is a fully discrete scheme and it preserves many good properties of the original NLSE in the discretized level, including mass conservation and time symmetry as well as dispersion relation. When $0<\sigma\leq\frac{1}{2} $, under the assumption of $ H^2 $-solution of the NLSE, 
we prove error bounds at $ O(\tau^{\frac{1}{2}+\sigma} + h^{1+2\sigma}) $ in $ L^2 $-norm without any CFL-type time step size restriction, which also fill the gap between the results in \cite{ignat2011,choi2021}. When $ \sigma \geq 1/2 $, under the assumption of $ H^2 $-solution again, we prove error bounds at $ O(\tau + h^{2}) $ and $ O(\tau^\frac{1}{2} + h) $ in $ L^2 $-norm and $ H^1 $-norm, respectively, with a very mild CFL-type time step size restriction, which generalize the result in \cite{ignat2011} to the mass-conservative fully discrete scheme. In addition, when $\frac{1}{2}<\sigma<1$ and under the assumption of $ H^3 $-solution, we show a new error bound at $ O(\tau^{\sigma} + h^{2\sigma}) $ in $ H^1 $-norm.

The rest of the paper is organized as follows. In Section 2, we present the time-splitting sine pseudospectral (TSSP) method, introduce a local regularization for the semi-smooth nonlinearity to be used for obtaining improved local truncation errors and state our main results. Section 3 is devoted to error estimates of the TSSP method for $ 0 < \sigma \leq 1/2 $ and Section 4 is devoted to error estimates for $ \sigma \geq 1/2 $. Numerical results are reported in Section 5 to confirm the error estimates. Finally some conclusions are drawn in Section 6. Throughout the paper, we adopt the standard Sobolev spaces as well as the corresponding norms, and denote by $ C $ a generic positive constant independent of the mesh size $ h $, time step $ \tau $, and by $ C(\alpha) $ a generic positive constant depending on $ \alpha $. The notation $ A \lesssim B $ is used to represent that there exists a generic constant $ C>0 $, such that $ |A| \leq CB $.

\section{Numerical methods and main results}

\subsection{The TSSP method}
We shall use the Lie-Trotter splitting method for the temporal discretization and use the sine pseudospectral method for the spatial discretization. The operator splitting technique is based on the decomposition of the flow of \cref{NLSE}
\begin{equation}
	\partial_t \psi = A(\psi) + B(\psi),
\end{equation}
where
\begin{equation}\label{eq:AB}
	A(\psi) = i\Delta \psi, \quad B(\psi) = B_1(\psi) + B_2(\psi) := -i V \psi -i f(|\psi|^2) \psi,
\end{equation}
into two sub-problems. The first one is
\begin{equation}
	\left\{
	\begin{aligned}
		&\partial_t \psi(\vx, t) = A(\psi) = i \Delta \psi(\vx, t), \quad \vx \in \Omega, \quad t>0, \\
		&\psi(\vx, 0) = \psi_0(\vx), \quad \vx \in \overline{\Omega},
	\end{aligned}
	\right.
\end{equation}
which can be formally integrated exactly in time as
\begin{equation}
	\psi(\cdot, t) = e^{i t \Delta} \psi_0(\cdot), \quad t \geq 0.
\end{equation}
The second one is to solve
\begin{equation}\label{eq:sub2}
	\left\{
	\begin{aligned}
		&\partial_t \psi(\vx, t) = B(\psi) = -iV(\vx)\psi(\vx, t) -i f(|\psi(\vx, t)|^2)\psi(\vx, t), \quad \vx \in \Omega,\  t>0, \\
		&\psi(\vx, 0) = \psi_0(\vx), \quad \vx \in \overline{\Omega},
	\end{aligned}
	\right.
\end{equation}
which, by using the fact
$|\psi(\vx, t)|=|\psi_0(\vx)|$ for $t\ge0$, can be integrated exactly in time as
\begin{equation}\label{eq:phi_B_def}
	\psi(\vx, t) = \Phi_B^t (\psi_0) := e^{- i t V(\vx)} \pbtt{\psi_0(\vx)}, \quad \vx \in \overline{\Omega}, \quad t \geq 0,
\end{equation}
where
\begin{equation}\label{eq:phi_B_2_def}
	\quad \pbtt{z} = ze^{-i t f(|z|^2)}, \quad z \in \C, \quad t\ge0.
\end{equation}
{In fact, in the second subproblem \cref{eq:sub2}, the operator $B(\psi) = -i(V + f(|\psi_0|^2))\psi$ becomes a bounded linear operator. }

Choose a time step size  $ \tau > 0 $, denote time steps as $ t_k = k \tau $ for $ k = 0, 1, ... $, and let $ \psi^{[k]}:=\psi^{[k]}(\vx) $ be the approximation of $ \psi(\vx, t_k) $ for $ k \geq 0 $. Then a first order semi-discretization of the NLSE \cref{NLSE} via the Lie-Trotter splitting is given as:
\begin{equation}\label{eq:semi_discretization}
	\psi^{[k+1]} = e^{i \tau \Delta} \Phi_B^\tau(\psi^{[k]}),
\end{equation}
with $ \psi^{[0]}(\vx) = \psi_0(\vx) $ for $\vx\in\overline{\Omega}$. 

Then we discretize \eqref{eq:semi_discretization} in space  by the sine pseudospectral method to obtain a full discretization for the NLSE \cref{NLSE}. For simplicity of notations,  here we only present the spatial discretization in 1D (taking $ \Omega = (a, b) $), and the generalization to higher dimensions is straightforward. Choose a mesh size $ h = (b-a)/N $ with $ N $ being a positive integer and denote grid points as
\begin{equation*}
	x_j = a+jh, \quad j = 0,1, \cdots, N.
\end{equation*}
Define the index sets
\begin{equation*}
	\mathcal{T}_N = \{1, 2, \cdots, N-1\}, \quad \mathcal{T}_N^0 = \{0, 1, \cdots, N\},
\end{equation*}
and denote
\begin{equation}
	\begin{aligned}
		X_N = \text{span}\left\{\sin(\mu_l(x-a)): l \in \mathcal{T}_N\right\}, \quad \mu_l = \frac{\pi l}{b-a},  \label{eq:varphi}\\
		Y_N = \left\{ v=(v_0, v_1, \cdots, v_N)^T \in \C^{N+1}: v_0 = v_N = 0 \right\}.
	\end{aligned}
\end{equation}
We define the $ l^p (1 \leq p \leq \infty) $ norm on $ Y_N $ as
\begin{align*}
	\| v \|_{l^p} = \left( h\sum_{j=0}^{N-1} |v_j|^p \right)^\frac{1}{p}, \quad 1 \leq p < \infty, \qquad
	\| v \|_{l^\infty} = \max_{0 \leq j \leq N-1} |v_j|, \quad v \in Y_N.
\end{align*}
We shall sometimes identify a function $ \phi(\cdot) \in C_0(\overline{\Omega}) $ with a vector $ \phi = (\phi_0, \phi_1,\cdots,$ $\phi_N)^T \in Y_N $ with $ \phi_j = \phi(x_j) $ and then the discrete norm $ \| \cdot \|_{l^p} $
can also be defined on $ X_N $. For $ v \in Y_N $, we define the forward finite difference operator as
\begin{equation}\label{eq:dxp_def}
	(\dxp v)_j = \dxp v_j = \frac{v_{j+1} - v_j}{h}, \quad 0 \leq j \leq N-1.
\end{equation}
Let $ P_N:L^2(\Omega) \rightarrow X_N $ be the standard $ L^2 $ projection onto $ X_N $ and $ I_N: Y_N  \rightarrow X_N $ be the standard sine interpolation operator as
\begin{equation}\label{eq:PN_IN_def}
\begin{aligned}
	&(P_N v)(x) = \sum_{l \in \mathcal{T}_N} \widehat v_l \sin(\mu_l(x-a)), \\
	&(I_N w)(x) = \sum_{l \in \mathcal{T}_N} \widetilde w_l \sin(\mu_l(x-a)),
\end{aligned}
\qquad x \in \overline{\Omega} = [a, b],
\end{equation}
where $ v \in L^2(\Omega) $, $ w \in Y_N $, and
\begin{equation}\label{phi_tilde} 
\begin{aligned}
	&\widehat{v}_l = \frac{2}{b-a} \int_a^b v(x) \sin(\mu_l(x-a)) \rmd x, \\
	&\widetilde{w}_l = \frac{2}{N} \sum_{j \in \mathcal{T}_N} w_j \sin(j \pi l / N),
\end{aligned}
\qquad l \in \mathcal{T}_N.
\end{equation}

Let $ \psi^k_j $ be the numerical approximations of $ \psi(x_j, t_k) $ for $ j \in \mathcal{T}_N^0 $ and $ k \geq 0 $, and denote $ \psi^k:=(\psi_0^k,\psi_1^k,\ldots,\psi_N^k)^T\in Y_N$. Then the time-splitting sine pseudospectral (TSSP) method for discretizing the NLSE \cref{NLSE} can be given for $ k \geq 0 $ as
\begin{equation}\label{full_discretization_scheme}
	\begin{aligned}
		&\psi^{(1)}_j = e^{- i \tau (V(x_j) + f(|\psi^k_j|^2))} \psi^k_j, \\
		&\psi^{k+1}_j = \sum_{l \in \mathcal{T}_N} e^{- i \tau \mu_l^2} \widetilde{(\psi^{(1)})}_l \sin(\mu_l(x_j-a)),
	\end{aligned}
	\qquad j \in \mathcal{T}_N^0, 
\end{equation}
where $ \psi^0_j = \psi_0(x_j) $ for $ j \in \mathcal{T}_N^0 $.

Let $ \Phi^\tau: X_N \rightarrow X_N $ be the numerical integrator defined as
\begin{equation}\label{eq:phitau_def}
	\Phi^\tau(\phi) = e^{i\tau\Delta} I_N \Phi_B^\tau (\phi), \quad \phi \in X_N,
\end{equation}
where $ \Phi_B^\tau $ is defined in \cref{eq:phi_B_def}. Then one has
\begin{equation}\label{recurrence_relation}
	\begin{aligned}
		&I_N \psi^{k+1} = \Phi^\tau(I_N \psi^k), \quad k \geq 0, \\
		&I_N \psi^0 = I_N \psi_0.
	\end{aligned}
\end{equation}

\begin{remark}
In applications, the NLSE \cref{NLSE} can also be discretized by the Lie-Trotter splitting via a different order as:
\begin{equation}\label{eq:semi_discretization-new}
	\psi^{[k+1]} =  \Phi_B^\tau(e^{i \tau \Delta}\psi^{[k]}), \qquad k\geq0.
\end{equation}
Then a full discretization can be obtained straightforward by using the sine pseudospectral method in space.
\end{remark}

\subsection{A local regularization for $ f(\rho) = \beta \rho^\sigma \ {(\beta \in \R)} $}

When $ 0 < \sigma < 1 $ in \eqref{semi-smooth}, $f(\rho)$ is a semi-smooth function and
it is not differentiable at $\rho=0$. {Here, we want to regularize it to obtain higher order local error estimates later. }
Following the regularization methods used in \cite{bao2022} for
the logarithmic {S}chr\"{o}dinger equation, we regularize the
semi-smooth nonlinearity $ f(\rho)$ only locally in a small region near $\rho=0$.
Taking $0<\varepsilon\ll 1$ as a regularization parameter, we approximate $f(\rho)$
locally in the region $ \{\rho < \vep^2 \} $ by a polynomial and leave it unchanged in $ \{\rho \geq \vep^2\} $, i.e.
\begin{equation}\label{eq:f_reg_def}
	\fe(\rho) =
	\left\{\begin{aligned}
		&f(\rho), &\rho \geq \vep^2 \\
		&\rho Q_\vep(\rho), &0\leq \rho < \vep^2,
	\end{aligned}
	\right.
\end{equation}
where $ Q_\vep(\rho) $ is a polynomial with degree at most $ 3 $ such that
\begin{equation}\label{eq:interpolation_condition}
	\fe \in C^{3}([0, \infty)).
\end{equation}
Note that $ \fe $ given by \cref{eq:f_reg_def} is uniquely determined by the interpolation conditions \cref{eq:interpolation_condition} and it satisfies $ \fe(0) = f(0) = 0 $. Actually, the explicit formula of $ Q_\vep(\rho) $ can be given as
\begin{equation}\label{Qe}
	Q_\vep(\rho) = \beta \vep^{2\sigma-2} \sum_{j=0}^3 \binom{j-\sigma}{j} \left( 1 - \frac{\rho}{\vep^2} \right)^j, \quad 0 \leq \rho < \vep^2.
\end{equation}

In fact, $ \fe \in C^{3}([0, \infty))$ can be regarded as a local regularization of the semi-smooth nonlinearity $f(\rho)\in C^{0}([0, \infty))$, which has much better regularity near $\rho=0$. For $f_\vep$, we have the following estimates. 
\begin{lemma}\label{lem:prop_fe}
	When $ 0 < \sigma < 1 $, we have
	\begin{equation}\label{f_1}
		|\fe(\rho)| + |\rho \fe^\prime(\rho)| \leq C_1 \rho^\sigma, \qquad \rho\ge0,
	\end{equation}
	\begin{equation}\label{f_2}
		|\sqrt{\rho} \fe^\prime(\rho)| + |\rho^\frac{3}{2} \fe^{\prime\prime}(\rho)| \leq C_2
		\left \{
		\begin{aligned}
			&\frac{1}{\vep^{1-2\sigma}}, && 0<\sigma \leq \frac{1}{2}, \\
			&{\rho^{\sigma-\frac{1}{2}}}, && \frac{1}{2}<\sigma<1 ,
		\end{aligned}
		\right., \quad \rho\ge0,
	\end{equation}
	and
	\begin{equation} \label{f_3}
		|\fe^\prime(\rho)| +|\rho \fe^{\prime\prime}(\rho)| + |\rho^2 \fe^{\prime\prime\prime}(\rho)| \leq \frac{C_3}{\vep^{2-2\sigma}},\quad \rho\ge0,
	\end{equation}
	where $ C_1 $, $ C_2 $ and $ C_3 $ depend exclusively on $ \sigma $ and $ \beta $.
\end{lemma}

\begin{proof}
	When $ \rho \geq \vep^2 $, by \cref{eq:f_reg_def}, we have $ f^{(k)}(\rho) = \fe^{(k)}(\rho) $ for $ 0 \leq k \leq 3 $, and \cref{f_1,f_2,f_3} follows immediately from $ f(\rho) = \beta \rho^\sigma $ and $ 0<\vep<1 $.
	
	In the following, we assume that $ 0 \leq \rho < \vep^2 $. From \cref{Qe}, we easily obtain that
	\begin{equation}\label{Qe_bound}
		|Q^{(k)}_\vep(\rho)| \lesssim \vep^{2\sigma - 2-2k}, \quad 0 \leq \rho < \vep^2, \quad 0 \leq k \leq 3.
	\end{equation}
	From \cref{eq:f_reg_def}, using \cref{Qe_bound}, one gets
	\begin{equation}
		|\fe(\rho)| \leq \rho | Q_\vep(\rho)| \lesssim \rho\vep^{2\sigma - 2} = \rho^\sigma \left(\frac{\rho}{\vep^2}\right)^{1-\sigma} \leq \rho^\sigma.
	\end{equation}
	Similarly, one has
	\begin{equation}
		|\rho \fe^\prime(\rho)| \leq |\rho^2 Q_\vep\p(\rho)| + |\rho Q_\vep(\rho)| \lesssim \left (\frac{\rho}{\vep^2}+1 \right  ) \rho \vep^{2\sigma-2} \leq 2\rho^\sigma,
	\end{equation}
	which proves \cref{f_1}.
	
	Recalling \cref{eq:f_reg_def}, using \cref{Qe_bound}, one gets, {when $ 0<\sigma<1 $,
	\begin{subequations}
		\begin{equation}\label{2.26a}
			|\sqrt{\rho} \fe\p(\rho)| \leq \rho^\frac{1}{2} \left( |Q_\vep(\rho)| + \rho |{Q_\vep}\p(\rho)|  \right) \lesssim \vep \left( \vep^{2\sigma-2} + \vep^2 \vep^{2\sigma-4} \right) \lesssim \vep^{2\sigma - 1},
		\end{equation}
		\begin{align}\label{2.26b}
			|\sqrt{\rho} \fe\p(\rho)| 
			\leq \rho^\frac{1}{2} \left( |Q_\vep(\rho)| + \rho |{Q_\vep}\p(\rho)|  \right)
			&\lesssim \rho^\frac{1}{2} \vep^{2\sigma - 2} \notag \\
			&= \rho^{\sigma - \frac{1}{2}} \left(\frac{\rho}{\vep^2}\right)^{1-\sigma} 
			\leq \rho^{\sigma - \frac{1}{2}}.
		\end{align}
	\end{subequations}
	Using \cref{2.26a} when $0<\sigma\leq 1/2$ and using \cref{2.26b} when $1/2<\sigma<1$, we obtain the desired estimate for $|\sqrt{\rho} \fe\p(\rho)|$.} The estimate of $ |\rho^\frac{3}{2} \fe^{\prime\prime}(\rho)| $ can be obtained similarly, which completes the proof of \cref{f_2}.
	
	For \cref{f_3}, using \cref{Qe_bound} again, one has
	\begin{equation}
		|\fe^\prime(\rho)| \leq |Q_\vep(\rho)| + \rho|Q_\vep\p(\rho)| \lesssim \vep^{2\sigma - 2} + \vep^2 \vep^{2\sigma - 4} \lesssim \vep^{2\sigma-2}.
	\end{equation}
	The estimate of $ |\rho \fe^{\prime\prime}(\rho)| $ and $ |\rho^2 \fe^{\prime\prime\prime}(\rho)| $ can be obtained similarly, which completes the proof of \cref{f_3}.
\end{proof}

\begin{corollary}\label{lem:Be}
	When $ 0 < \sigma \leq 1/2 $, we have
	\begin{align}
		&\| \fe(|v|^2)v \|_{L^2} \leq C_1(\| v \|_{L^\infty}) \| v \|_{L^2}, \quad v \in L^\infty(\Omega), \label{eq:Be_L2}\\
		&\| \fe(|v|^2)v \|_{H^1} \leq C_2(\| v \|_{L^\infty}) \| v \|_{H^1}, \quad v \in H^1(\Omega) \cap L^\infty(\Omega), \label{eq:Be_H1}\\
		&\| \fe(|v|^2)v \|_{H^2} \leq \frac{C_3\left(\|  v \|_{H^2}\right)}{\vep^{1-2\sigma}}, \quad v \in H^2(\Omega).  \label{eq:Be_H2}
	\end{align}
	When $ 0 < \sigma <1 $, we have
	\begin{equation}\label{eq:Be_H3}
		\| \fe(|v|^2)v \|_{H^3} \leq \frac{C_4\left(\|  v \|_{H^3}\right)}{\vep^{2-2\sigma}}, \quad v \in H^3(\Omega).
	\end{equation}
\end{corollary}

\begin{proof}
	By \cref{f_1}, one has
	\begin{equation}
		\| \fe(|v|^2)v \|_{L^2} \leq \| \fe(|v|^2) \|_{L^\infty} \| v \|_{L^2} \lesssim \| v \|_{L^\infty}^{2\sigma} \| v \|_{L^2},
	\end{equation}
	which proves \cref{eq:Be_L2}.
	
	By direct calculation, using \cref{f_1}, one gets
	\begin{align}
		\left \| \nabla \left (\fe(|v|^2)v \right ) \right \|_{L^2}
		&= \left \| \fe(|v|^2) \nabla v + \fe\p(|v|^2) v ( v\nabla \overline{v} + \overline{v} \nabla v) \right \|_{L^2}\notag \\
		&\leq \left (\| \fe(|v|^2) \|_{L^\infty} + \| \fe\p(|v|^2)v^2 \|_{L^\infty} + \| \fe\p(|v|^2)|v|^2 \|_{L^\infty}\right ) \| \nabla v \|_{L^2} \notag\\
		&\lesssim \| v \|_{L^\infty}^{2\sigma} \| v \|_{H^1},
	\end{align}
	which shows \cref{eq:Be_H1}.
	
	To show \cref{eq:Be_H2}, we note that
	\begin{align}\label{parparBe}
		\partial_{jk}(\fe(|v|^2) v)
		&=\partial_j \left [(\fe(|v|^2) + \fe\p(|v|^2)|v|^2) \partial_k v + \fe\p(|v|^2)v^2 \partial_k \overline{v} \right ]\notag \\
		&=(2\fe\p(|v|^2) + \fe\pp(|v|^2)|v|^2) \partial_j |v|^2 \partial_k v + (\fe(|v|^2) + \fe\p(|v|^2)|v|^2) \partial_{jk} v \notag\\
		&\quad + \fe\pp(|v|^2) v^2 \partial_j |v|^2 \partial_k \overline{v} + 2 \fe\p(|v|^2)v \partial_j v \partial_k \overline{v} + \fe\p(|v|^2)v^2 \partial_{jk} \overline{v},
	\end{align}
	where $ \partial_j = \partial_{x_j} $ and $ \partial_{jk} = \partial_{x_j} \partial_{x_k} $ for $ 1 \leq j, k \leq d $. Here we adopt the notations
$\vx=x_1$ (or $x$) when $d=1$, $\vx=(x_1,x_2)^T$ (or $(x,y)^T$) when $d=2$,
and $\vx=(x_1,x_2,x_3)^T$ (or $(x,y,z)^T$) when $d=3$.
From Lemma \ref{lem:prop_fe}, using \cref{f_1,f_2} and noting that $ |\partial_j |v|^2| \leq 2 |v|\, |\partial_j v|$, one gets
	\begin{align}\label{parparBe_bound}
		\left | \partial_{jk}(\fe(|v|^2) v) \right |
		&\lesssim \left (\fe\p(|v|^2)|v| + \fe\pp(|v|^2)|v|^3 \right ) |\partial_j v|\, |\partial_k v| \notag\\
		&\quad + \left ( \fe(|v|^2) + \fe\p(|v|^2)|v|^2 \right ) |\partial_{jk} v| \notag\\
		&\lesssim \frac{|\partial_j v|\, |\partial_k v|}{\vep^{1-2\sigma}} + |v|^{2\sigma} |\partial_{jk} v|,
	\end{align}
	which, by using H\"older's inequality and Sobolev embedding {$H^2 \hookrightarrow W^{1,4}$ which holds for $d=1, 2, 3$}, yields
	\begin{equation}\label{parparBe_norm}
		\| \partial_{jk}(\fe(|v|^2) v) \|_{L^2} \lesssim \frac{\|\partial_j v\|_{L^4} \|\partial_k v\|_{L^4}}{\vep^{1-2\sigma}} + \|v\|_{L^\infty}^{2\sigma} \|\partial_{jk} v \|_{L^2} \leq \frac{C(\| v \|_{H^2})}{\vep^{1-2\sigma}},
	\end{equation}
	which implies \cref{eq:Be_H2}.
	
	Following \cref{lem:prop_fe}, noting \eqref{parparBe_bound} and
\eqref{parparBe_norm} and using \cref{f_3}, we can similarly obtain \cref{eq:Be_H3} and the details are omitted here for brevity.
\end{proof}

\begin{lemma}\label{eq:f-fe1}
	When $ 0< \sigma < 1 $, we have
	\begin{equation*}
		| f(\rho) - \fe(\rho) | \leq C \vep^{2 \sigma} \mathbbm{1}_{\rho < \vep^2}, \quad \rho \geq 0.
	\end{equation*}
\end{lemma}

\begin{proof}
	Recalling \cref{eq:f_reg_def}, we have
	\begin{equation}
		| f(\rho) - \fe(\rho) | = 0, \quad \rho \geq \vep^2,
	\end{equation}
	and, by \cref{semi-smooth,f_1},
	\begin{equation}
		| f(\rho) - \fe(\rho) | \leq | f(\rho) | + | \fe(\rho) | \lesssim \rho^\sigma \leq \vep^{2\sigma}, \quad 0 \leq \rho < \vep^2,
	\end{equation}
	which completes the proof.
\end{proof}

\subsection{Main results}
Let $ T_\text{max} $ be the maximal existing time for the solution of the NLSE
\eqref{NLSE} with \eqref{init} and \eqref{bound} and take $ 0 < T < T_\text{max} $
be a fixed time. Based on the known existence and regularity results ({see Remark 4.8.7 (iii) in \cite{cazenave2003} or Theorem II in \cite{kato1987}}) for the solution of \eqref{NLSE}, we make the assumption that the solution $\psi$ satisfies $ \psi \in C([0, T]; H_0^1(\Omega) \cap H^2(\Omega)) \cap C^1([0, T]; L^2(\Omega)) $ such that
\begin{equation}\label{A}
\| \psi \|_{L^\infty([0, T]; H^2)} + \| \partial_t \psi \|_{L^\infty([0, T]; L^2)} \lesssim 1. \tag{A}
\end{equation}
{Note that the solution to \cref{NLSE} that satisfies \cref{A} must be unique \cite{henning2017}.}

Define
\begin{equation}\label{M2}
	M_2 := \max\left\{ \| \psi \|_{L^\infty([0, T]; H^2)}, \| \psi \|_{L^\infty([0, T]; L^\infty)}, \| V \|_{H^2} \right\},
\end{equation}
and assume the following time step size restriction ($h<1$)
\begin{equation}\label{cfl}
\tau \lesssim \left\{
\begin{aligned}
&1, &\qquad d=1,\\
&\frac{1}{|\ln h|^2}, &\qquad d = 2, \\
&h, \quad &\qquad  d=3.
\end{aligned}
\right. \tag{B}
\end{equation}
For the TSSP method \cref{full_discretization_scheme}, we can establish the following error estimates.

\begin{theorem}\label{thm:sigma<1/2}
	When $ 0 < \sigma \leq 1/2 $, under the assumptions $V \in H^2(\Omega)$ and \cref{A}, for $ 0 < \tau < 1 $ and $ 0<h<1 $, we have
	\begin{equation}\label{eq:thm1}
\| \psi(\cdot,t_k)- I_N \psi^k \|_{L^2} \lesssim \tau^{1/2+\sigma} + h^{1+2\sigma}, \quad 0 \leq k \leq \frac{T}{\tau}.
	\end{equation}

\end{theorem}

\begin{corollary}\label{cor:1D}
When $ d = 1 $ and $ 0 < \sigma \leq 1/2 $, under the following much weaker assumptions
\[ V \in H^1(\Omega), \qquad  \psi \in C([0, T]; H_0^1(\Omega)),
\]
we have  for $ 0 < \tau < 1 $ and $ 0<h<1 $,
\begin{equation}\label{eq:thm1_1D}
\|\psi(\cdot,t_k)- I_N \psi^k \|_{L^2} \lesssim \tau^{1/2} + h, \quad 0 \leq k \leq \frac{T}{\tau}.
\end{equation}
\end{corollary}

\begin{theorem}\label{thm:sigma>1/2H2}
When $ \sigma \geq 1/2 $, under the assumptions $ V \in H^2(\Omega) \cap W^{1, \infty}(\Omega) $ and \cref{A}, there exist $ \tau_0>0 $ and $ h_0>0 $ sufficiently small and depending on $ M_2 $, $ \| V \|_{W^{1, \infty}} $ and $ T $ such that for $ \tau \leq \tau_0 $ and $ h \leq h_0 $ satisfying \cref{cfl}, we have
	\begin{equation}\label{eq:thm2_part1}
		\begin{aligned}
			&\| \psi(\cdot,t_k) -I_N \psi^{k}  \|_{L^2} \lesssim \tau + h^2, \quad \| \psi^k \|_{l^\infty} \leq 1+M_2\\
			&\| \psi(\cdot,t_k)-I_N \psi^{k}  \|_{H^1} \lesssim \tau^\frac{1}{2} + h, \quad 0 \leq k \leq \frac{T}{\tau}.
		\end{aligned}
	\end{equation}
Moreover, when $ 1/2<\sigma<1 $, under the additional assumptions that $ V \in H^{3}(\Omega) $, $ \nabla V \in H_0^1(\Omega) $ and $ \psi \in C([0, T]; H_*^3(\Omega)) \cap C^1([0, T]; H^1(\Omega)) $, we have
\begin{equation}\label{eq:thm2_part2}
\| \psi(\cdot,t_k)-I_N \psi^{k}  \|_{H^1} \lesssim \tau^\sigma + h^{2\sigma}, \quad 0 \leq k \leq \frac{T}{\tau},
	\end{equation}
where $H_*^3(\Omega):=\{\phi\in H^3(\Omega)\ | \ \phi(\vx)|_{\partial\Omega}=\Delta \phi(\vx)|_{\partial\Omega}=0 \}$.
\end{theorem}

\begin{remark}
	When $ \sigma \geq 1 $, under the same assumptions as those for \cref{eq:thm2_part2}, one can obtain the following error bound for the TSSP method \cref{full_discretization_scheme} as
	\begin{equation*}
		\| \psi(\cdot,t_k)-I_N \psi^{k}  \|_{H^1} \lesssim \tau + h^2, \quad 0 \leq k \leq \frac{T}{\tau}.
	\end{equation*}
\end{remark}
	


\section{Proof of \cref {thm:sigma<1/2} for the case $ 0 < \sigma \leq 1/2  $}
Throughout this section, we assume that $ V \in H^2(\Omega) $, $ 0 < \sigma \leq 1/2  $ and the assumption \cref{A}.
\subsection{Some estimates for the operator $B$ }

For the operator $ B $ defined in \cref{eq:AB}, we have

\begin{lemma}\label{lem:B_1}
	Let $ v \in H^1(\Omega) $ such that $ \| v \|_{L^\infty} \leq M $. When $ \sigma > 0 $, we have
\begin{align}
	&\| \B{v} \|_{L^2} \leq C_1(M, \| V \|_{L^\infty}) \| v \|_{L^2}, \label{eq:B_L2}\\
  	&\| \B{v} \|_{H^1} \leq \| v \|_{H^1} \left\{\begin{array}{ll}
  C_2(M, \| V \|_{H^1}), &d=1,\\
  C_2(M, \| V \|_{W^{1, 4}}), &d=2,3.
  \end{array}\right.
\end{align}
\end{lemma}
\begin{proof}
From the definition of $ B $ in \cref{eq:AB}, we have
\begin{equation}
	\| B(v) \|_{L^2} \leq \| V \|_{L^\infty} \| v \|_{L^2} + C(\| v \|_{L^\infty}) \| v \|_{L^2},
\end{equation}
which implies \cref{eq:B_L2}.

Introduce a continuous function $ G: \C \rightarrow \C $ as
\begin{equation}\label{eq:G_def}
	G(z) = \left\{
	\begin{aligned}
		& f^\prime(|z|^2) z^2=\beta \sigma |z|^{2\sigma-2} z^2, &z \neq 0, \\
		&0, &z=0,
	\end{aligned}
	\right. \qquad z \in \C,
\end{equation}
and {note that $ f^\prime(|z|^2) |z|^2 = \sigma f(|z|^2) $ for $z \in \C$. Further note that}
\begin{equation}\label{eq:bound_G}
	f(|z|^2) + |G(z)| \lesssim |z|^{2\sigma}, \qquad z \in \C, \qquad \sigma > 0.
\end{equation}
Direct calculation yields
\begin{align}\label{eq:partial_B}
		\nabla B(v)
		&= -i\left[ V \nabla v + v \nabla V  + f(|v|^2) \nabla v + f\p(|v|^2) v ( v\nabla \overline{v} + \overline{v} \nabla v) \right] \notag\\
		&= -i\left[ V \nabla v + v \nabla V  + (1+\sigma) f(|v|^2) \nabla v + \G{v} \nabla\overline{v} \right],
\end{align}
where $ G(v)(\vx) := G(v(\vx)) $ for $ \vx \in \Omega $. From \cref{eq:partial_B}, using H\"{o}lder's inequality and noticing \cref{eq:bound_G}, we obtain
\begin{equation}\label{eq:est_dB}
\| \nabla B(v) \|_{L^2}
\lesssim \| V \|_{L^\infty} \| \nabla v \|_{L^2} + \| v \|_{L^\infty}^{2\sigma} \| \nabla v \|_{L^2} +
\left\{\begin{array}{ll}
	  \| v \|_{L^\infty} \| \nabla V \|_{L^2}, &d=1, \\
	  \| v \|_{L^4} \| \nabla V \|_{L^4}, &d=2, 3,
\end{array}\right.,
\end{equation}
where different estimates are used for $ v \nabla V $ for $ d=1 $ and $ d=2,3 $.
Thus we have, {by Sobolev embedding $H^1 \hookrightarrow L^\infty$ when $d=1$ and $H^1 \hookrightarrow L^4$ when $d=2, 3$,}
\begin{equation*}
\| \nabla B(v) \|_{L^2}\leq C(\| v \|_{L^\infty}) \| v \|_{H^1}+
 \| v \|_{H^1} \left\{\begin{array}{ll}
 C(\| V \|_{H^1}), &d=1,\\
 C(\| V \|_{W^{1, 4}}), &d=2,3,\\
 \end{array}\right.
\end{equation*}
which completes the proof.
\end{proof}

\begin{lemma}\label{lem:diff_B_L2}
	Let $ v, w \in L^\infty(\Omega) $ such that $ \| v \|_{L^\infty} \leq M $ and $ \| w \|_{L^\infty} \leq M $. When $ \sigma > 0 $, we have
	\begin{equation*}
		\| \B{v} - \B{w} \|_{L^2} \leq C(M, \| V \|_{L^\infty}) \| v -w \|_{L^2}.
	\end{equation*}
\end{lemma}

\begin{proof}
	Recalling \cref{eq:AB}, we have
	\begin{equation}\label{diff_B}
		\| \B{v} - \B{w} \|_{L^2} \leq \| V \|_{L^\infty} \| v-w \|_{L^2} + \| f(|v|^2)v - f(|w|^2)w \|_{L^2}.
	\end{equation}
	For any $ z_1, z_2 \in {\mathbb C} $, let $ z_\theta = (1-\theta)z_1 + \theta z_2 $ and let $ \gamma(\theta) = f(|z_\theta|^2)z_\theta $ for $ 0 \leq \theta \leq 1 $, we have
	\begin{equation}\label{f(z)z}
		f(|z_2|^2)z_2 - f(|z_1|^2)z_1  = \gamma(1) - \gamma(0) = \int_0^1 \gamma^\prime(\theta) \rmd \theta.
	\end{equation}
	Recalling \cref{semi-smooth,eq:G_def}, we have
	\begin{equation}\label{diff_gamma}
		\gamma^\prime(\theta) = (1+\sigma)f(|z_\theta|^2)(z_2 - z_1) + G(z_\theta) \overline{(z_2 - z_1)}.
	\end{equation}
	Plugging \cref{diff_gamma} into \cref{f(z)z}, noticing \cref{eq:bound_G}, we have
	\begin{equation}\label{eq:lip_f(z)z}
		|f(|z_1|^2)z_1 - f(|z_2|^2)z_2| \leq \sup_{0 \leq \theta \leq 1} |\gamma^\prime(\theta)| \lesssim \max\{|z_1|, |z_2|\}^{2\sigma} |z_1 - z_2|.
	\end{equation}
	Thus we have
	\begin{equation}
		\| f(|v|^2)v - f(|w|^2)w \|_{L^2} \leq C(\max\{\| v \|_{L^\infty}, \| w \|_{L^\infty}\}) \| v - w \|_{L^2},
	\end{equation}
	which plugged into \cref{diff_B} completes the proof.
\end{proof}

Let $ dB(\cdot)[\cdot] $ be the G\^ateaux derivative defined as
\begin{equation}\label{gateaux}
dB(v)[w]:= \lim_{\varepsilon \rightarrow 0} \frac{B(v + \varepsilon w) - B(v)}{\vep}, 
\end{equation}
{where the limit is taken for real $ \vep $, and we identify $\C$ with $\R^2$ to be consistent with the complex valued setting (see also the appendix in \cite{kato1987}). }
Then we have
\begin{lemma}\label{lem:dB_L2}
	Let $ v \in L^\infty(\Omega) $ such that $ \| v \|_{L^\infty} \leq M $ and $ w \in L^2(\Omega) $. When $ \sigma > 0 $, we have
\begin{equation*}
\| dB(v) [w] \|_{L^2} \leq C(M, \| V \|_{L^\infty}) \| w \|_{L^2}.
	\end{equation*}
\end{lemma}

\begin{proof}
Plugging \eqref{eq:AB} into \eqref{gateaux}, we obtain ({see (4.26) in \cite{bao2022}})
\begin{align}\label{gateaux-new}
dB(v)[w]
&= -i V w + d{B_{2}}(v)[w] \notag \\
&=-i \left[ V w + (1+\sigma) f(|v|^2) w + \G{v} \overline{w} \right],
\end{align}
where $ G $ is defined in \cref{eq:G_def}. From \cref{gateaux-new}, noting \cref{eq:bound_G}, we have
\begin{equation*}
	\| dB(v)[w] \|_{L^2} \leq \| V \|_{L^\infty} \| w \|_{L^2} + C(\| v \|_{L^\infty}) \| w \|_{L^2},
\end{equation*}
which concludes the proof.
\end{proof}



\begin{lemma}\label{lem:diff_Phi_B_2_sigma<1/2}
	 When $ 0 < \sigma \leq 1/2 $, we have
	\begin{equation*}
		| \Phi_{B_2}^\tau (z_1) - \Phi_{B_2}^\tau (z_2) | \leq \left(1 + C \tau\right) | z_1 - z_2 |, \qquad z_1, z_2 \in \C,
	\end{equation*}
	where $ \Phi_{B_2}^\tau (z) = z e^{- i \tau f(|z|^2)} $ in \cref{eq:phi_B_2_def} and $ C = 2\sigma|\beta| \min\{|z_1|, |z_2|\}^{2\sigma} $.
\end{lemma}

\begin{proof}
	Without loss of generality, we assume that $ | z_2 | \leq | z_1 | $. If $ z_2 = 0 $, the conclusion follows immediately. In the following, we assume that $ z_2 \neq 0 $. Then, by noting that $ |1 - e^{i \theta}| \leq |\theta| $ for all $ \theta \in \R $,
	we have
	\begin{align}\label{eq:diff_phi_B_2}
			| \Phi_{B_2}^\tau (z_1) - \Phi_{B_2}^\tau (z_2) |
			&= |z_1 e^{- i \tau f(|z_1|^2)} - z_2 e^{- i \tau f(|z_2|^2)}| \notag\\
			&\leq | z_1 - z_2 | + |z_2| \left| 1 - e^{- i \tau (f(|z_1|^2)-f(|z_2|^2))} \right| \notag\\
			&\leq | z_1 - z_2 | + \tau |z_2| \left|f(|z_1|^2)-f(|z_2|^2) \right|.
	\end{align}
	When $ 0 < \sigma \leq 1/2 $, since $ 0<|z_2|\leq|z_1| $, by the mean value theorem and the definition of $ f $ in \cref{semi-smooth}, we have
	\begin{align}\label{eq:diff_f_sigma<1/2}
			\left|f(|z_1|^2)-f(|z_2|^2) \right|
			&= |\beta|  \left ||z_1|^{2\sigma} - |z_2|^{2\sigma} \right | \notag\\
			&\leq \frac{2\sigma|\beta||z_1 - z_2|}{\min\{|z_1|, |z_2|\}^{1-2 \sigma}} = 2\sigma|\beta|\frac{|z_1 - z_2|}{|z_2|^{1-2 \sigma}}.
	\end{align}
	Plugging \cref{eq:diff_f_sigma<1/2} into \cref{eq:diff_phi_B_2}, we get the desired result immediately.
\end{proof}

\subsection{Local truncation error}
{In this subsection, we shall prove the local truncation error estimates for the TSSP \cref{full_discretization_scheme} in 1D, which can be directly generalized to 2D and 3D.} {With the regularized function $ \fe$ introduced in Section 2.2, we can obtain $\sigma$ sensitive estimates as follows. }
\begin{lemma}\label{lem:fractional_error_1}
	Let $ \phi \in X_N $ such that $ \| \phi \|_{H^2} \leq M $ and let $ 0<\tau<1 $ and $ 0 < h < 1 $. Assume that $ V \in H^2(\Omega) $. When $ 0<\sigma\leq 1/2 $, we have
	\begin{align}
		&\| (I-e^{i \tau \Delta}) P_N B(\phi) \|_{L^2} \leq C_1(M, \| V \|_{H^2}) \tau^{1/2+\sigma}, \label{eq:free_schrodinger_error_1}\\
		&\| I_N B(\phi) - P_N B(\phi) \|_{L^2} \leq C_2(M, \| V \|_{H^2}) h^{1+2\sigma}. \label{eq:inter_error_1}
	\end{align}
\end{lemma}

\begin{proof}
	Recalling the standard estimates that (see, e.g., \cite{book_spectral,bao2014,bao2019})
	\begin{align}
		&\| v - P_N v \|_{L^2} \lesssim h^2 |v|_{H^2}, \quad \| I_N v - P_N v \|_{L^2} \lesssim h^2 |v|_{H^2}, \label{eq:proj_inter_error_L2}\\
		&\| v - e^{i t \Delta} v \|_{L^2} \lesssim t \| v \|_{H^2}, \quad v \in H_0^1(\Omega) \cap H^2(\Omega), \label{eq:linear_flow_error_L2}
	\end{align}
	noting that $ H^2(\Omega) $ is an algebra when $ 1 \leq d \leq 3 $, we have
	\begin{equation}\label{Vphi_reg_e}
		\| (I-e^{i \tau \Delta}) (V \phi)  \|_{L^2} \lesssim \tau \| V \|_{H^2} \| \phi \|_{H^2}, \quad \| (I_N - P_N) (V \phi)  \|_{L^2} \lesssim h^2 \| V \|_{H^2} \| \phi \|_{H^2}.
	\end{equation}
	According to \cref{eq:AB}, it remains to show \cref{eq:free_schrodinger_error_1,eq:inter_error_1} with $ f(|\phi|^2)\phi $ replacing $ B(\phi) $. Using the regularized function $ \fe $ defined in \cref{eq:f_reg_def} with $ 0 < \vep \ll 1 $ and the triangle inequality, we have
	\begin{align}\label{f-fe+fe}
			&\| (I-e^{i \tau \Delta}) (f(|\phi|^2)\phi) \|_{L^2} \notag\\
			&\leq \| (I-e^{i \tau \Delta}) (f(|\phi|^2)\phi - \fe(|\phi|^2)\phi) \|_{L^2} + \| (I-e^{i \tau \Delta}) (\fe(|\phi|^2)\phi) \|_{L^2}.
	\end{align}
	From \cref{f-fe+fe}, using $ \| (I-e^{i \tau \Delta}) v \|_{L^2} \leq 2 \| v \|_{L^2} $ for the first term and \cref{eq:linear_flow_error_L2} for the second term, we have
	\begin{equation}\label{f_decomp}
		\| (I-e^{i \tau \Delta}) (f(|\phi|^2)\phi) \|_{L^2} \lesssim \| f(|\phi|^2)\phi - \fe(|\phi|^2)\phi \|_{L^2} + \tau \| \fe(|\phi|^2)\phi \|_{H^2}.
	\end{equation}
	By \cref{eq:f-fe1} and \cref{eq:Be_H2}, we have
	\begin{align}
		&\| f(|\phi|^2)\phi - \fe(|\phi|^2)\phi \|_{L^2} \lesssim \vep^{2\sigma} \| \phi \mathbbm{1}_{|\phi| < \vep} \|_{L^2} \leq |\Omega|^\frac{1}{2} \vep^{1+2\sigma}, \label{L2e_reg}\\
		&\| \fe(|\phi|^2)\phi \|_{H^2} \leq \frac{C(M)}{\vep^{1-2\sigma}}. \label{H2b}
	\end{align}
	Plugging \cref{L2e_reg,H2b} into \cref{f_decomp}, we have
	\begin{equation*}
		\| (I-e^{i \tau \Delta}) (f(|\phi|^2)\phi) \|_{L^2} \leq C(M)\inf_{0<\vep<1} \left(\vep^{1+2\sigma} + \frac{\tau}{\vep^{1-2\sigma}} \right) \leq C(M)\tau^{1/2+\sigma},
	\end{equation*}
	which combined with \cref{Vphi_reg_e} yields \cref{eq:free_schrodinger_error_1}.
	
	Then we shall prove \cref{eq:inter_error_1}. Similar to \cref{f-fe+fe,f_decomp}, using the triangle inequality, the $ L^2 $-projection property of $ P_N $, \cref{L2e_reg}, \cref{eq:proj_inter_error_L2} and \cref{eq:Be_H2}, we have
	\begin{align}\label{f_decomp2}
			&\| (I_N - P_N) (f(|\phi|^2)\phi) \|_{L^2} \notag\\
			&\leq \| (I_N - P_N) (f(|\phi|^2)\phi - \fe(|\phi|^2)\phi) \|_{L^2} + \| (I_N - P_N) (\fe(|\phi|^2)\phi) \|_{L^2} \notag\\
			&\leq \| I_N(f(|\phi|^2)\phi - \fe(|\phi|^2)\phi) \|_{L^2} + \| P_N(f(|\phi|^2)\phi - \fe(|\phi|^2)\phi) \|_{L^2} \notag\\
			&\quad + h^2 \| \fe(|\phi|^2)\phi \|_{H^2} \notag\\
			&\leq \| I_N(f(|\phi|^2)\phi - \fe(|\phi|^2)\phi) \|_{L^2} + |\Omega|^\frac{1}{2} \vep^{1+2\sigma} + h^2 \frac{C(M)}{\vep^{1-2\sigma}}.
	\end{align}
	By Parseval's identity,
	\begin{equation}\label{eq:I_N_L2}
		\| I_N v \|_{L^2} = \sqrt{ h\sum_{j=1}^{N-1} |v(x_j)|^2 } \leq \sqrt{ h\sum_{j=1}^{N-1} \| v \|_{l^\infty}^2 } \leq |\Omega|^\frac{1}{2} \| v \|_{l^\infty}, \quad v \in C_0(\overline{\Omega}),
	\end{equation}
	which implies, by using \cref{eq:f-fe1} again,
	\begin{align}\label{est_I_N_reg_error}
			\| I_N(f(|\phi|^2)\phi - \fe(|\phi|^2)\phi) \|_{L^2}
			&\leq |\Omega|^\frac{1}{2} \| (f(|\phi|^2) - \fe(|\phi|^2)) \phi \|_{l^\infty} \notag\\
			&\leq |\Omega|^\frac{1}{2} \vep^{2\sigma} \| \phi \mathbbm{1}_{|\phi| < \vep} \|_{l^\infty} \leq |\Omega|^\frac{1}{2} \vep^{1+2\sigma}.
	\end{align}
	Plugging \cref{est_I_N_reg_error} into \cref{f_decomp2}, we have
	\begin{equation*}
		\| (I_N - P_N) (f(|\phi|^2)\phi) \|_{L^2} \leq C(M) \inf_{0<\vep<1} \left(\vep^{1+2\sigma} + \frac{h^2}{\vep^{1-2\sigma}} \right) \leq C(M)h^{1+2\sigma},
	\end{equation*}
	which completes the proof.
\end{proof}

Now we are able to show the local truncation error of the TSSP method.
\begin{proposition}[local truncation error]\label{prop:local_error_sigma<1/2}
	Assume that $ V \in H^2(\Omega) $. Under the assumption \cref{A}, for $ 0 \leq k \leq T/\tau-1 $, we have
	\begin{equation*}
		\| P_N \psi(\cdot, t_{k+1}) - \Phi^\tau( P_N \psi(\cdot, t_k) ) \|_{L^2(\Omega)} \leq C(M_2) \tau \left( \tau^{\frac{1}{2}+\sigma} + h^{1+2\sigma} \right).
	\end{equation*}
\end{proposition}

\begin{proof}
	For the simplicity of notations, we define $ v(t) = \psi(t_k + t) := \psi(\cdot, t_k + t) $ for $ 0 \leq t \leq \tau $ and $ \vini := v(0) = \psi(t_k) $. By Sobolev embedding {$H^2 \hookrightarrow L^\infty$}, noting the boundedness of $ e^{it\Delta} $ and $ P_N $, we have
	\begin{align}
		&\| e^{i s \Delta} v(t) \|_{L^\infty} \lesssim \| e^{i s \Delta} v(t) \|_{H^2} = \| v(t) \|_{H^2} \leq M_2, \label{Linfty_bound_expit}\\
		&\| P_N v(t) \|_{L^\infty} \lesssim \| P_N v(t) \|_{H^2} \leq \| v(t) \|_{H^2} \leq M_2, \quad 0 \leq s, t \leq \tau. \label{Linfty_bound_P_N}
	\end{align}
 	{By variation of constant formula (see (4.24)-(4.25) in \cite{bao2022})}
	\begin{align}\label{eq:duhamel}
		\psi(t_{k+1})
		&= e^{i \tau \Delta} \vini + \int_0^\tau e^{i (\tau - s )\Delta} B(e^{is \Delta} \vini) \rmd s \notag\\
		&\quad + \int_0^\tau \int_0^s e^{i (\tau - s) \Delta} dB(e^{i(s-\sigma)\Delta}v(\sigma)) [e^{i(s - \sigma)\Delta} B(v(\sigma))] \rmd \sigma \rmd s,
	\end{align}
	where $ dB(\cdot)[\cdot] $ is the G\^ateaux derivative defined in \cref{gateaux}. Applying $ P_N $ on both sides of \cref{eq:duhamel}, {noting that $e^{i\tau\Delta}$ and $P_N$ commute \cite{review_2013}}, one gets
	\begin{align}
		P_N \psi(t_{k+1})
		&=e^{i \tau \Delta} P_N \vini + \int_0^\tau e^{i (\tau - s )\Delta} P_N B(e^{is \Delta} \vini) \rmd s \notag\\
		&\quad + \int_0^\tau \int_0^s e^{i (\tau - s) \Delta} P_N\left( dB(e^{i(s-\sigma)\Delta}v(\sigma)) [e^{i(s - \sigma)\Delta} B(v(\sigma))] \right) \rmd \sigma \rmd s. \label{localerror1}
	\end{align}
	From \cref{eq:phi_B_def}, recalling that $ \vini = \psi(t_k) $, we have
	\begin{equation}\label{phi_tau}
		\Phi^\tau(P_N \psi(t_k)) = e^{i \tau \Delta} I_N \Phi_B^\tau(P_N \vini).
	\end{equation}
	Applying the first-order Taylor expansion ({see proof of Theorem 4.2 in \cite{bao2022}})
	\begin{equation}
		\Phi_B^\tau(w) = w + \tau B(w) + \tau^2 \int_0^1 (1-\theta) dB(\Phi_B^{\theta \tau}(w)) [B(\Phi_B^{\theta \tau}(w))] \rmd \theta
	\end{equation}
	for $ w = P_N \vini $ and plugging it into \cref{phi_tau}, we have
	\begin{align}
		\Phi^\tau(P_N \psi(t_k))
		&=e^{i \tau \Delta} P_N \vini + \tau e^{i \tau \Delta} I_N B(P_N \vini)\nonumber \\
		&\quad + \tau^2 e^{i \tau \Delta} I_N \left(\int_0^1 (1-\theta)  \left( dB(\Phi_B^{\theta \tau}(P_N \vini)) [B(\Phi_B^{\theta \tau}(P_N \vini))] \right) \rmd \theta \right). \label{localerror2}
	\end{align}
	Subtracting \eqref{localerror2} from \eqref{localerror1}, we have
	\begin{equation}\label{eq:error_decomp}
		P_N \psi(t_{k+1}) - \Phi^\tau( P_N \psi(t_k) ) = e_1 - e_2 + e_3,
	\end{equation}
	where
	\begin{align}
		&e_1 =  \int_0^\tau \int_0^s e^{i (\tau - s) \Delta} P_N\left( dB(e^{i(s-\sigma)\Delta}v(\sigma)) [e^{i(s - \sigma)\Delta} B(v(\sigma))] \right) \rmd \sigma \rmd s, \label{e1} \\
		&e_2 =  \tau^2 e^{i \tau \Delta} I_N \left(\int_0^1 (1-\theta) \left( dB(\Phi_B^{\theta \tau}(P_N \vini)) [B(\Phi_B^{\theta \tau}(P_N \vini))] \right) \rmd \theta \right), \label{e2}\\
		&e_3 = \int_0^\tau e^{i (\tau - s )\Delta} P_N B(e^{is \Delta} \vini) \rmd s - \tau e^{i \tau \Delta} I_N B(P_N \vini). \label{e3}
	\end{align}
	
 	Next, we shall first estimate $ e_1 $ and $ e_2 $. Noticing the property of $ e^{it\Delta} $ and $ P_N $, using \cref{lem:dB_L2} and \cref{Linfty_bound_expit}, we have
 	\begin{align}\label{e_1_dB_L2}
 		&\left\| e^{i (\tau - s) \Delta} P_N \left( dB(e^{i(s-\sigma)\Delta}v(\sigma)) [e^{i(s - \sigma)\Delta} B(v(\sigma))] \right) \right\|_{L^2} \notag\\
 		&\leq \| dB(e^{i(s-\sigma)\Delta}v(\sigma)) [e^{i(s - \sigma)\Delta} B(v(\sigma))] \|_{L^2} \notag\\
 		&\leq C(\| V \|_{L^\infty}, \| e^{i(s-\sigma)\Delta}v(\sigma) \|_{L^\infty}) \| e^{i(s - \sigma)\Delta} B(v(\sigma)) \|_{L^2} \notag\\
 		&\leq C(M_2) \| B(v(\sigma)) \|_{L^2}.
 	\end{align}
 	From \cref{e1}, using \cref{e_1_dB_L2} and \cref{eq:B_L2}, we get
	\begin{align}\label{e1_L2_1}
			\| e_1 \|_{L^2}
			&\leq \int_0^\tau \int_0^s \left\| e^{i (\tau - s) \Delta} P_N \left( dB(e^{i(s-\sigma)\Delta}v(\sigma)) [e^{i(s - \sigma)\Delta} B(v(\sigma))] \right) \right\|_{L^2} \rmd \sigma \rmd s \notag\\
			&\leq C(M_2) \int_0^\tau \int_0^s \| B(v(\sigma)) \|_{L^2} \rmd \sigma \rmd s \leq C(M_2) \tau^2 \max_{0 \leq \sigma \leq \tau} \| B(v(\sigma)) \|_{L^2} \notag\\
			&\leq C(M_2) \tau^2 C(M_2) \max_{0 \leq \sigma \leq \tau} \| v(\sigma) \|_{L^2} \leq C(M_2) \tau^2.
	\end{align}
	From \cref{eq:phi_B_def} and \cref{eq:AB}, using \cref{Linfty_bound_P_N}, one gets,
	\begin{equation}\label{Linfty_Phi_B}
		\begin{aligned}
			&\|\Phi_B^{\theta \tau}(P_N \vini) \|_{L^\infty} = \| P_N \vini \|_{L^\infty} \leq C(M_2), \qquad 0 \leq \theta \leq 1, \\
			&\| B(\Phi_B^{\theta \tau}(P_N \vini)) \|_{L^\infty} \leq C(\| V \|_{L^\infty}, \|P_N \vini\|_{L^\infty}) \|P_N \vini\|_{L^\infty}\leq C(M_2).
		\end{aligned}
	\end{equation}
	From \cref{gateaux-new}, noticing \cref{eq:bound_G}, one easily gets
	\begin{equation}
		\| dB(w_1)[w_2] \|_{L^\infty} \leq C(\| V \|_{L^\infty}, \| w_1 \|_{L^\infty}) \| w_2 \|_{L^\infty}, \quad w_1, w_2 \in L^\infty(\Omega),
	\end{equation}
	which combined with \cref{eq:I_N_L2} and \cref{Linfty_Phi_B}, yields the estimate for $ e_2 $ in \cref{e2} as
	\begin{align}\label{e2_L2_1}
			\| e_2 \|_{L^2}
			&\leq \tau^2 \left \| I_N \left(\int_0^1 (1-\theta) \left( dB(\Phi_B^{\theta \tau}(P_N \vini)) [B(\Phi_B^{\theta \tau}(P_N \vini))] \right) \rmd \theta \right) \right \|_{L^2} \notag \\
			&\leq \tau^2 |\Omega|^\frac{1}{2} \max_{0 \leq \theta \leq 1} \left \| dB(\Phi_B^{\theta \tau}(P_N \vini))[B(\Phi_B^{\theta \tau}(P_N \vini))] \right \|_{l^\infty} \notag \\
			&\leq \tau^2 |\Omega|^\frac{1}{2} C\left (\| V \|_{L^\infty}, \|\Phi_B^{\theta \tau}(P_N \vini) \|_{L^\infty} \right ) \| B(\Phi_B^{\theta \tau}(P_N \vini)) \|_{L^\infty} \notag \\
			&\leq C(M_2) \tau^2.
	\end{align}

	Then we shall estimate $ e_3 $ in \cref{e3}, which can be written as
	\begin{equation*}
		e_3 = \int_0^\tau \left[ e^{i (\tau - s )\Delta} P_N B(e^{is \Delta} \vini) - e^{i \tau \Delta} I_N B(P_N \vini) \right] \rmd s,
	\end{equation*}
	which yields
	\begin{equation}\label{e3_est_1}
		\| e_3 \|_{L^2} \leq \tau \max_{0 \leq s \leq \tau} \| e^{i (\tau - s )\Delta} P_N B(e^{is \Delta} \vini) - e^{i \tau \Delta} I_N B(P_N \vini) \|_{L^2}.
	\end{equation}
	Using standard properties of $ e^{i t \Delta} $ and $ P_N $, one gets
	\begin{align}\label{e3_decomp_1}
			&\| e^{i (\tau - s )\Delta} P_N B(e^{is \Delta} \vini) - e^{i \tau \Delta} I_N B(P_N \vini) \|_{L^2} \notag\\
			&=\| P_N B(e^{is \Delta} \vini) - e^{i s \Delta} I_N B(P_N \vini) \|_{L^2} \notag\\
			&\leq \| P_N B(e^{is \Delta} \vini) - P_N B(\vini) \|_{L^2} + \| P_N B(\vini) - P_N B(P_N\vini) \|_{L^2} \notag\\
			&\quad+ \| P_N B(P_N\vini) - e^{is\Delta} P_N B(P_N\vini) \|_{L^2} \notag\\
			&\quad+ \| e^{is\Delta} P_N B(P_N\vini) - e^{i s \Delta} I_N B(P_N \vini) \|_{L^2} \notag\\
			&\leq \| B(e^{is \Delta} \vini) - B(\vini) \|_{L^2} + \| B(\vini) - B(P_N\vini) \|_{L^2} \notag\\
			&\quad+ \| (I-e^{is\Delta}) P_N B(P_N\vini) \|_{L^2} + \| (P_N - I_N) B(P_N\vini) \|_{L^2} \notag\\
			&=: \| e_3^1 \|_{L^2} + \| e_3^2 \|_{L^2} + \| e_3^3 \|_{L^2} + \| e_3^4 \|_{L^2}.
	\end{align}
	For $ e_3^1 $ and $ e_3^2 $ in \cref{e3_decomp_1}, using \cref{lem:diff_B_L2}, recalling \cref{eq:proj_inter_error_L2,eq:linear_flow_error_L2,Linfty_bound_expit,Linfty_bound_P_N}, we obtain
	\begin{equation}\label{e3_12}
		\begin{aligned}
			&\| e_3^1 \|_{L^2} = \| B(e^{is \Delta} \vini) - B(\vini) \|_{L^2} \leq C(M_2) \| (I-e^{is \Delta}) \vini \|_{L^2} \leq C(M_2) \tau, \\
			&\| e_3^2 \|_{L^2} = \| B(\vini) - B(P_N\vini) \|_{L^2} \leq C(M_2) \| \vini - P_N \vini \|_{L^2} \leq C(M_2) h^2.
		\end{aligned}
	\end{equation}
	For $ e_3^3 $ and $ e_3^4 $ in \cref{e3_decomp_1}, using \cref{lem:fractional_error_1}, we get
	\begin{equation}\label{e3_34}
		\begin{aligned}
			&\| e_3^3 \|_{L^2} = \| (I-e^{is\Delta}) P_N B(P_N\vini) \|_{L^2} \leq C(M_2) \tau^\frac{1+2\sigma}{2}, \\
			&\| e_3^4 \|_{L^2} = \| (I_N-P_N) B(P_N\vini) \|_{L^2} \leq C(M_2) h^{1+2\sigma}.
		\end{aligned}
	\end{equation}
	Plugging \cref{e3_12,e3_34} into \cref{e3_decomp_1}, and noticing \cref{e3_est_1}, we get
	\begin{equation}\label{e3_L2_1}
		\| e_3 \|_{L^2} \leq C(M_2) \tau \left( \tau^\frac{1+2\sigma}{2} + h^{1+2\sigma} \right).
	\end{equation}
	Combing \cref{e1_L2_1,e2_L2_1,e3_L2_1}, and noting \cref{eq:error_decomp}, we get the desired result.
\end{proof}

\begin{remark}\label{rem:local_error_1D}
	The proof of \cref{prop:local_error_sigma<1/2} can be generalized to 2D and 3D directly. Moreover, in 1D, under much weaker assumption that $ V \in H^1(\Omega) $ and $ \psi \in C([0, T]; H_0^1(\Omega)) $, by using Sobolev embedding $ H^1 \hookrightarrow L^\infty $ and the estimates (see, e.g., \cite{bao2019,bao2022})
	\begin{equation}
		\| v - e^{it\Delta} v \|_{L^2} \lesssim \sqrt{\tau} \| v \|_{H^1}, \quad \| v - P_N v \|_{L^2} \lesssim h |v|_{H^1}, \qquad v \in H_0^1(\Omega),
	\end{equation}
	and following the proof of \cref{prop:local_error_sigma<1/2}, we can obtain
	\begin{equation}\label{eq:local_error_1D}
		\| P_N \psi(t_{k+1}) - \Phi^\tau( P_N \psi(t_k) ) \|_{L^2(\Omega)} \leq C \tau \left( \sqrt{\tau} + h \right),
	\end{equation}
	where $ C $ depends on $ \| V \|_{H^1} $ and $ \| \psi \|_{L^\infty([0, T]; H^1)} $.
\end{remark}

\subsection{Unconditional $ L^2 $-stability and proof of \cref{thm:sigma<1/2}}
We shall show the unconditional $ L^2 $-stability of the numerical flow by using \cref{lem:diff_Phi_B_2_sigma<1/2}. With the estimate of the local truncation error and the unconditional $ L^2 $-stability of the numerical flow, we are able to obtain the error estimates.

\begin{proposition}[unconditional $ L^2 $-stability]\label{prop:stability of Phi_AB_sigma<1/2}
	Let $ v \in X_N $ and $ w \in X_N $ such that $ \min\{\| v \|_{L^\infty}, \| w \|_{L^\infty}\} \leq M $. When $ 0 < \sigma \leq 1/2 $, we have
	\begin{equation*}
		\| \Phi^\tau(v) - \Phi^\tau(w) \|_{L^2} \leq  (1 + C(M) \tau) \| v - w \|_{L^2},
	\end{equation*}
	where {$\Phi^\tau$ is defined in \cref{eq:phitau_def}} and $ C(M) \sim M^{2\sigma} $.
\end{proposition}

\begin{proof}
	Recalling \cref{eq:phitau_def}, noting that $ e^{i \tau \Delta} $ preserves the $ L^2 $ norm, one gets
	\begin{align}\label{reduce_L2_stability}
		\| \Phi^\tau(v) - \Phi^\tau(w) \|_{L^2}
		&= \| e^{i \tau \Delta} I_N \Phi_B^\tau (v) -  e^{i \tau \Delta} I_N \Phi_B^\tau (w) \|_{L^2} \notag\\
		&= \| I_N \Phi_B^\tau (v) - I_N \Phi_B^\tau (w) \|_{L^2}.
	\end{align}
	From \cref{reduce_L2_stability}, by \cref{eq:I_N_L2,lem:diff_Phi_B_2_sigma<1/2}, noting that $ I_N $ is an identity on $ X_N $, and recalling \cref{eq:phi_B_def}, we have
	\begin{align}\label{eq:stability_L2_proof}
			&\| I_N \Phi_B^\tau (v) - I_N \Phi_B^\tau (w) \|_{L^2}^2 \notag\\
			&= h \sum_{j=1}^{N-1} | \Phi_B^\tau(v)(x_j) - \Phi_B^\tau(w)(x_j) |^2 \notag\\
			&= h \sum_{j=1}^{N-1} \left| e^{-i\tau V(x_j)} \Phi_{B_2}^\tau(v)(x_j) - e^{-i\tau V(x_j)} \Phi_{B_2}^\tau(w)(x_j) \right|^2 \notag\\
			&= h \sum_{j=1}^{N-1} \left| \Phi_{B_2}^\tau(v)(x_j) - \Phi_{B_2}^\tau(w)(x_j) \right|^2 \notag\\
			&\leq (1 + C(M) \tau)^2  h \sum_{j=1}^{N-1} \left| v(x_j) - w(x_j) \right|^2 \notag\\
			&= (1 + C(M) \tau)^2 \| I_N v -I_N w \|_{L^2}^2 \notag\\
			&= (1 + C(M) \tau)^2 \| v - w \|_{L^2}^2.
	\end{align}
The proof is completed.
\end{proof}

\begin{remark}
	In the error estimates, $ v $ and $ w $ in \cref{prop:stability of Phi_AB_sigma<1/2} are related to the exact solution and the numerical solution, respectively. Hence, to control the constant $ C(M) $ in \cref{prop:stability of Phi_AB_sigma<1/2},  we can assume bound of the exact solution and thus get rid of the a priori estimate of the numerical solution, which explains why \cref{prop:stability of Phi_AB_sigma<1/2} is called the unconditional $ L^2 $-stability.
\end{remark}

\begin{proof}[Proof of \cref{thm:sigma<1/2}]
	Under the assumption \cref{A}, using \cref{eq:proj_inter_error_L2}, one gets
	\begin{equation}
		\| \psi(\cdot, t_k) - P_N \psi(\cdot, t_k) \|_{L^2} \leq C(M_2) h^2.
	\end{equation}
	Hence, it suffices to estimate $ e^k := I_N \psi^k - P_N \psi(\cdot, t_k) \in X_N $ for $ 0 \leq k \leq T/\tau $. By \cref{recurrence_relation}, for $ 0 \leq k \leq T/\tau - 1 $, one has
	\begin{align}\label{eq:error_propagation}
			\| e^{k+1} \|_{L^2}
			&= \| I_N \psi^{k+1} - P_N \psi(\cdot, t_{k+1}) \|_{L^2} = \| \Phi^\tau (I_N \psi^{k}) - P_N \psi(\cdot, t_{k+1}) \|_{L^2} \notag\\
			&\leq \| \Phi^\tau (I_N \psi^{k}) - \Phi^\tau (P_N \psi(\cdot, t_k)) \|_{L^2} + \| \Phi^\tau (P_N \psi(\cdot, t_k)) - P_N \psi(\cdot, t_{k+1}) \|_{L^2}.
	\end{align}
	By \cref{prop:stability of Phi_AB_sigma<1/2,prop:local_error_sigma<1/2}, noting that $ \| P_N \psi(\cdot, t_k) \|_{L^\infty} \lesssim \| P_N \psi(\cdot, t_k) \|_{H^2} \leq \| \psi(\cdot, t_k) \|_{H^2} \leq M_2 $, one has
	\begin{equation*}
		\| e^k \|_{L^2(\Omega)} \leq {(1+C(M_2) \tau)} \| e^{k-1} \|_{L^2(\Omega)} + C(M_2)  \tau \left( \tau^{1/2+\sigma} + h^{1+2\sigma} \right), \quad 1 \leq k \leq T/\tau.
	\end{equation*}
	It follows from the discrete Gronwall's inequality and $ \| e^0 \|_{L^2} = \| I_N \psi_0 - P_N \psi_0 \| \leq C(M_2) h^2 $ that
	\begin{equation*}
		\| e^k \|_{L^2(\Omega)} \leq C(T, M_2) \left( \tau^\frac{1+2\sigma}{2} + h^{1+2\sigma} \right), \quad 0 \leq k \leq T/\tau,
	\end{equation*}
	which completes the proof.
\end{proof}

The proof of \cref{cor:1D} follows the proof of \cref{thm:sigma<1/2} by replacing \cref{prop:local_error_sigma<1/2} with \cref{eq:local_error_1D} and we shall omit it for brevity.

\section{Proof of \cref{thm:sigma>1/2H2} for the case $ \sigma \geq 1/2 $}
In this section, we assume that $ V \in H^2(\Omega) \cap W^{1, \infty}(\Omega) $, $ \sigma \geq 1/2 $ and the assumption \cref{A}. The assumption $ V \in W^{1, \infty}(\Omega) $ is only used in \cref{prop:stability of Phi_AB_sigma>1/2} and can be obtained from $ V \in H^2(\Omega) $ in 1D or $ V \in H^3(\Omega) $ in 2D and 3D. {Also, we shall use the equivalent norm $\| \nabla \cdot \|_{L^2}$ on $H_0^1(\Omega)$ to avoid frequent use of Poincar\'e inequality. }

\subsection{Some estimates for the operator B}
\begin{lemma}\label{lem:B_2}
	Let $ v \in H^2(\Omega) $ such that $ \| v \|_{H^2} \leq M $. When $ \sigma \geq 1/2 $, we have
	\begin{equation*}
		\| \B{v} \|_{H^2(\Omega)} \leq  C(M, \| V \|_{H^2}).
	\end{equation*}
\end{lemma}

\begin{proof}
	Recalling \cref{eq:AB}, noting that $ H^2(\Omega) $ is an algebra when $ 1 \leq d \leq 3 $, we have
	\begin{equation}\label{B_H2}
		\| B(v) \|_{H^2} \leq \| V v \|_{H^2} + \| f(|v|^2)v \|_{H^2} \leq \| V \|_{H^2} \| v \|_{H^2} + \| f(|v|^2)v \|_{H^2}.
	\end{equation}
	When $ \sigma \geq 1/2 $, recalling \cref{semi-smooth,eq:G_def}, by similar calculation as \eqref{parparBe} and \eqref{parparBe_bound} and noting \cref{eq:bound_G} as well as
	\begin{equation} \label{eq:bound_f_2sgima-1}
		\left | f\p(|z|^2)z \right | + \left | f\p(|z|^2)\overline{z} \right  | + \left | f\pp(|z|^2)z^3 \right | + \left | f\pp(|z|^2)z^2 \overline{z} \right | \lesssim |z|^{2\sigma-1}, \  z \in \C, \  \sigma \geq \frac{1}{2},
	\end{equation}
	we have
	\begin{equation}
		\left | \partial_{jk}(f(|v|^2) v) \right | \lesssim |v|^{2\sigma} |\partial_{jk} v| + |v|^{2\sigma-1}|\partial_j v| \,|\partial_k v|,
	\end{equation}
	which yields, by Sobolev embedding {$H^2 \hookrightarrow W^{1,4}$ for $d=1,2,3$}, that
	\begin{equation}\label{parital_jk f(v)v}
		\| \partial_{jk}(f(|v|^2) v) \|_{L^2} \lesssim \| v \|_{L^\infty}^{2\sigma} \| \partial_{jk} v \|_{L^2} + \|v\|_{L^\infty}^{2\sigma - 1} \| \nabla v\|_{L^4}^2 \leq C(M).
	\end{equation}
	Combing \cref{parital_jk f(v)v} and \cref{lem:B_1}, noting \cref{B_H2}, we obtain the desired result.
\end{proof}

\begin{lemma}\label{lem:diff_B_H1}
	Let $ v, w \in H^2(\Omega) $ such that $ \| v \|_{H^2} \leq M $ and $ \| w \|_{H^2} \leq M $. When $ \sigma \geq 1/2 $, we have
	\begin{equation*}\label{eq:diff_H1}
		\| \B{v} - \B{w} \|_{H^1} \leq C(M,\| V \|_{W^{1, 4}}) \| v - w \|_{H^1}.
	\end{equation*}
\end{lemma}
	
\begin{proof}
	From \cref{eq:partial_B}, one gets
	\begin{equation}\label{nabla_diff_B}
		\begin{aligned}
			\nabla \left (\B{v} - \B{w} \right )
			&= -i \left [ \nabla (V (v-w)) + i (1+\sigma) (f(|v|^2) \nabla v - f(|w|^2) \nabla w) \right. \\
			&\quad \left. + \G{v} \nabla\overline{v} - \G{w} \nabla\overline{w} \right ].
		\end{aligned}
	\end{equation}
	Using H\"older's inequality and Sobolev embedding {$H^1 \hookrightarrow L^4$ and $W^{1, 4} \hookrightarrow L^\infty$ (both hold for $d=1,2,3$)}, we have
	\begin{align}\label{eq:Vv}
			\| \nabla (V(v - w)) \|_{L^2}
			&\leq \| \nabla V \|_{L^4} \| v-w \|_{L^4} + \| V \|_{L^\infty} \| \nabla (v-w) \|_{L^2} \notag\\
			&\lesssim \| V \|_{W^{1, 4}} \| v-w \|_{H^1}.
	\end{align}
	By \cref{nabla_diff_B}, it remains to show that
	\begin{align}
		&\| f(|v|^2) \nabla v - f(|w|^2) \nabla w \|_{L^2} \leq C(M) \| v - w \|_{H^1}, \label{eq:diff_H1_1}\\
		&\| \G{v} \nabla \overline{v} - \G{w} \nabla \overline{w} \|_{L^2} \leq C(M) \| v - w \|_{H^1}. \label{eq:diff_H1_2}
	\end{align}
	When $ \sigma \geq 1/2 $, following the proof of \cref{eq:lip_f(z)z}, we have, for $ z_1, z_2 \in \C $,
	\begin{align}
		&| f(|z_1|^2) - f(|z_2|^2) | \lesssim  \max\{|z_1|, |z_2|\}^{2\sigma-1}  |z_1 - z_2|, \label{eq:diff_f}\\
		&| \G{z_1} - \G{z_2}| \lesssim \max\{|z_1|, |z_2|\}^{2\sigma-1} |z_1 - z_2|. \label{eq:diff_G}
	\end{align}
	Using \cref{eq:diff_f} and Sobolev embedding {$H^1 \hookrightarrow L^4$ and $H^2 \hookrightarrow L^\infty$}, we have
	\begin{align*}
		&\| f(|v|^2) \nabla v - f(|w|^2) \nabla w \|_{L^2}\\
		&\leq \| f(|v|^2) \nabla (v - w) \|_{L^2} + \| (f(|v|^2) - f(|w|^2)) \nabla w \|_{L^2} \\
		&\leq C(\| v \|_{L^\infty}) \| v - w \|_{H^1} + C(\max\{\| v \|_{L^\infty}, \| w \|_{L^\infty}\}) \| (v - w) \nabla w \|_{L^2} \\
		&\leq C(M) \| v - w \|_{H^1} + C(M) \| v - w \|_{L^4}\| \nabla w \|_{L^4} \\
		&\leq C(M) \| v - w \|_{H^1},
	\end{align*}
	which proves \cref{eq:diff_H1_1}. Similarly, we can prove \cref{eq:diff_H1_2}, which completes the proof.
\end{proof}

\begin{lemma}\label{lem:dB(v)w_2}
	Let $ v, w \in H^1(\Omega) \cap L^\infty(\Omega) $ such that $ \| v \|_{L^\infty} + \| v \|_{H^1} \leq M $ and $ \| w \|_{L^\infty} + \| w \|_{H^1} \leq M $. When $ \sigma \geq 1/2 $, we have
	\begin{equation*}
		\| dB(v) [w] \|_{H^1} \leq C(M, \| V \|_{W^{1, 4}}). 
	\end{equation*}
\end{lemma}

\begin{proof}
	From \cref{gateaux-new}, using \cref{eq:Vv}, we have
	\begin{align}\label{dbvw_H1}
			\| dB(v) [w] \|_{H^1}
			&\leq \| V w \|_{H^1} + (1+\sigma) \| f(|v|^2)w \|_{H^1} + \| G(v) \overline{w} \|_{H^1} \notag\\
			&\lesssim \| V \|_{W^{1, 4}} \| w \|_{H^1} +  \| f(|v|^2)w \|_{H^1} + \| G(v) \overline{w} \|_{H^1}.
	\end{align}
	When $ \sigma \geq 1/2 $, recalling \cref{eq:bound_f_2sgima-1}, we have
	\begin{align}\label{H1_fv}
			\| f(|v|^2) \|_{H^1}
			&= \| f(|v|^2) \|_{L^2} + \| \nabla f(|v|^2) \|_{L^2} \lesssim \| v \|_{L^\infty}^{2\sigma} +  \| f\p(|v|^2)v \nabla v \|_{L^2} \notag\\
			&\leq \| v \|_{L^\infty}^{2\sigma} +  \| v \|_{L^\infty}^{2\sigma-1} \| \nabla v \|_{L^2} \leq C(M). 
	\end{align}
	Similarly, one gets $ \| G(v) \|_{H^1} \leq C(M) $. Then using
	\begin{equation}\label{eq:bilinear_H1}
		\| u_1 u_2 \|_{H^1} \leq \| u_1 \|_{L^\infty} \| u_2 \|_{H^1} + \| u_2 \|_{L^\infty} \| u_1 \|_{H^1}, \quad u_1, u_2 \in H^1(\Omega) \cap L^\infty(\Omega),
	\end{equation}
	and recalling \cref{eq:bound_G}, we have
	\begin{align}
		&\| f(|v|^2)w \|_{H^1} \leq \| f(|v|^2) \|_{L^\infty} \| w \|_{H^1} + \| w \|_{L^\infty} \| f(|v|^2) \|_{H^1} \leq C(M), \label{fvw_H1}\\
		&\| G(v) \overline{w} \|_{H^1} \leq  \| G(v) \|_{L^\infty} \| \overline{w} \|_{H^1} + \| \overline{w} \|_{L^\infty} \| G(v) \|_{H^1} \leq C(M). \label{Gvw_H1}
	\end{align}
	Plugging \cref{fvw_H1,Gvw_H1} into \cref{dbvw_H1} yields the desired result.
\end{proof}

\begin{lemma}\label{lem:dBvv_H2}
	Let $ v, w \in H^2(\Omega) $ such that $ \| v \|_{H^2} \leq M $ and $ \| w \|_{H^2} \leq M $. If $ |w(\vx)| \leq C |v(\vx)| $ for all $ \vx \in \Omega $, when $ \sigma \geq 1/2 $, we have
	\begin{equation*}
		\| dB(v) [w] \|_{H^2} \leq C \left (M, \| V \|_{H^2} \right).
	\end{equation*}
\end{lemma}

\begin{proof}
	The proof can be obtained similarly as the proof of \cref{lem:B_2} and we shall omit it here for brevity.
\end{proof}

\begin{lemma}\label{lem:bound_nonlinear_flow}
	Let $ 0 < \tau < 1 $ and $ v \in X_N $ such that $ \| v \|_{L^\infty} \leq M $ and $ \| v \|_{H^2} \leq M_1 $. When $ \sigma > 0 $, we have
	\begin{equation}\label{eq:phi_B_H1}
		\| \Phi_{B}^\tau (v) \|_{H^1} \leq \left (1 + C_1 (M, \| V \|_{W^{1, 4}}) \tau \right ) \| v \|_{H^1(\Omega)},
	\end{equation}
	and when $ \sigma \geq 1/2 $, we have
	\begin{equation}\label{eq:phi_B_H2}
		\| \Phi_{B}^\tau (v) \|_{H^2} \leq C_2(M_1, \| V \|_{H^2}).
	\end{equation}
\end{lemma}

\begin{proof}
	Recalling that $ \Phi_{B}^\tau (v) = ve^{- i \tau (V+f(|v|^2))} $ in \cref{eq:phi_B_def}, the proof of \cref{eq:phi_B_H1} and \cref{eq:phi_B_H2} follows similarly from the proof of \cref{lem:B_1} and \cref{lem:B_2}, respectively.

\end{proof}

\begin{lemma}\label{lem:diff_phi_B_2_sigma>1/2}
	Let $ z_1, z_2 \in \C $. When $ \sigma \geq 1/2 $, one has
	\begin{equation*}
		| \Phi_{B_2}^\tau (z_1) - \Phi_{B_2}^\tau (z_2) | \leq (1 + C \tau) | z_1 - z_2 |,
	\end{equation*}
	where $ \Phi_{B_2}^\tau (z) = z e^{- i \tau f(|z|^2)} $ in \cref{eq:phi_B_2_def} and $ C \sim \max\{|z_1|, |z_2|\}^{2\sigma} $.
\end{lemma}

\begin{proof}
	The proof follows from the proof of \cref{lem:diff_Phi_B_2_sigma<1/2} by replacing \cref{eq:diff_f_sigma<1/2} with \cref{eq:diff_f}.
\end{proof}

\subsection{Local truncation error}
\begin{proposition}[local truncation error]\label{prop:local_error_sigma>1/2}
	Assume that $ 0 < \tau < 1 $, $ 0 < h < 1 $, $ V \in H^2 $ and $ \sigma \geq 1/2 $. Under the assumption \cref{A}, for $ 0 \leq k \leq T/\tau - 1 $, we have
	\begin{align}
		&\| P_N \psi(\cdot, t_{k+1}) - \Phi^\tau( P_N \psi(\cdot, t_k) ) \|_{L^2(\Omega)} \leq C_1(M_2) \tau \left( \tau + h^2 \right), \label{localerror:L^2_2}\\
		&\| P_N \psi(\cdot, t_{k+1}) - \Phi^\tau( P_N \psi(\cdot, t_k) ) \|_{H^1(\Omega)} \leq C_2(M_2) \tau \left( \tau^\frac{1}{2} + h \right). \label{localerror_H^1_1}
	\end{align}
\end{proposition}

\begin{proof}
	Following the notations in the proof of \cref{prop:local_error_sigma<1/2}, we let $ v(t) = \psi(\cdot, t_k + t) $ for $ 0 \leq t \leq \tau $ and $ \vini := v(0) = \psi(\cdot, t_k) $. When $ \sigma \geq 1/2 $, \cref{Linfty_bound_P_N,Linfty_bound_expit} are also valid and we have the same error decomposition \cref{eq:error_decomp}. When $ \sigma \geq 1/2 $, the $ L^2 $ estimate \cref{localerror:L^2_2} follows from the proof of \cref{prop:local_error_sigma<1/2} by replacing \cref{e3_34} with
	\begin{equation}\label{e3_34_sigma>1/2}
		\begin{aligned}
			&\| e_3^3 \|_{L^2} \lesssim \tau \| P_N B(P_N\vini) \|_{H^2} \leq \tau \| B(P_N\vini) \|_{H^2} \leq C(M_2) \tau, \\
			&\| e_3^4 \|_{L^2} \lesssim h^2 \| B(P_N\vini) \|_{H^2} \leq C(M_2) h^2,
		\end{aligned}
	\end{equation}
	where \cref{eq:linear_flow_error_L2}, \cref{eq:proj_inter_error_L2} and \cref{lem:B_2} are used.
	
	In the following, we shall show \cref{localerror_H^1_1}. Using Sobolev embedding {$H^2 \hookrightarrow L^\infty$}, the isometry property of $ e^{i t \Delta} $ and \cref{lem:B_1,lem:B_2}, one gets
	\begin{equation}\label{Bbound}
		 \begin{aligned}
		 	&\| e^{i(s - \sigma)\Delta} B(v(\sigma)) \|_{H^1} = \| B(v(\sigma)) \|_{H^1} \leq C(M_2), \\
		 	&\| e^{i(s - \sigma)\Delta} B(v(\sigma)) \|_{L^\infty} \lesssim \| e^{i(s - \sigma)\Delta} B(v(\sigma)) \|_{H^2} = \| B(v(\sigma)) \|_{H^2} \leq C(M_2).
		 \end{aligned}
	\end{equation}
	Recalling the boundedness of $ e^{it\Delta} $ and $ P_N $, using \cref{lem:dB(v)w_2}, noticing \cref{Linfty_bound_expit} and \cref{Bbound}, we have
	 \begin{align}\label{e_1_dB_H1}
			&\left\| e^{i (\tau - s) \Delta} P_N \left( dB(e^{i(s-\sigma)\Delta}v(\sigma)) [e^{i(s - \sigma)\Delta} B(v(\sigma))] \right) \right\|_{H^1} \notag\\
			&\leq \| dB(e^{i(s-\sigma)\Delta}v(\sigma)) [e^{i(s - \sigma)\Delta} B(v(\sigma))] \|_{H^1} \notag\\
			&\leq C(\| V \|_{W^{1, 4}}, \| e^{i(s-\sigma)\Delta}v(\sigma) \|_{L^\infty \cap H^1}, \| e^{i(s - \sigma)\Delta} B(v(\sigma)) \|_{L^\infty \cap H^1} ) \notag\\
			&\leq C(M_2),
	\end{align}
	which yields, for $ e_1 $ in \cref{e1},
	\begin{align}\label{e1_H1}
			\| e_1 \|_{H^1}
			&\leq \int_0^\tau \int_0^s \left\| e^{i (\tau - s) \Delta} P_N \left( dB(e^{i(s-\sigma)\Delta}v(\sigma)) [e^{i(s - \sigma)\Delta} B(v(\sigma))] \right) \right\|_{H^1} \rmd \sigma \rmd s \notag\\
			&\leq C(M_2) \tau^2.
	\end{align}
	For $ e_2 $ in \cref{e2}, using the estimate for $\phi \in H_0^1(\Omega) \cap H^2(\Omega)$, 
	{\begin{equation}
		\| I_N \phi \|_{H^1} \leq \| \phi \|_{H^1} + \| \phi - I_N \phi \|_{H^1} \lesssim \| \phi \|_{H^1} + h | \phi |_{H^2} \lesssim \| \phi \|_{H^2}, 
	\end{equation}}
	one gets
	\begin{align}\label{est:e2_2}
			\| e_2 \|_{H^1}
			&= \tau^2 \left \| I_N \left(\int_0^1 (1-\theta) \left( dB(\Phi_B^{\theta \tau}(P_N \vini)) [B(\Phi_B^{\theta \tau}(P_N \vini))] \right) \rmd \theta \right) \right \|_{H^1} \notag\\
			&\lesssim \tau^2 \|dB(\Phi_B^{\theta \tau}(P_N \vini)) [B(\Phi_B^{\theta \tau}(P_N \vini))] \|_{H^2}.
	\end{align}
	From \cref{est:e2_2}, noting that
	\begin{equation}
		\left | \left [B(\Phi_B^{\theta \tau}(P_N \vini))\right](\vx) \right | \lesssim \left | \Phi_B^{\theta \tau}(P_N \vini) (\vx) \right | = (P_N v_0)(\vx), \qquad \vx \in \Omega,
	\end{equation}
	and using \cref{lem:dBvv_H2}, we have
	\begin{equation}\label{e2_H1}
		\| e_2 \|_{H^1} \leq C(M_2) \tau^2.
	\end{equation}

	Then we shall estimate $ e_3 $ in \cref{e3}. Similar to \cref{e3_est_1,e3_decomp_1}, it suffices to bound the $ H^1 $-norm of the four terms $ e_3^j (1 \leq j \leq 4) $ defined in \cref{e3_decomp_1}. Using the standard estimates ({see, e.g., \cite{book_spectral,bao2019}}), 
	\begin{align}
		&\| \phi - P_N \phi \|_{H^1} \lesssim h |\phi|_{H^2}, \quad \| I_N \phi - P_N \phi \|_{H^1} \lesssim h |\phi|_{H^2}, \label{eq:proj_inter_error_H1}\\
		&\| \phi - e^{i t \Delta} \phi \|_{H^1} \lesssim \sqrt{t} \| \phi \|_{H^2}, \quad \phi \in H_0^1(\Omega) \cap H^2(\Omega), \label{eq:linear_flow_error_H1}
	\end{align}
	and \cref{lem:diff_B_H1,lem:B_2}, we have
	\begin{align}
		&\begin{aligned}\label{e3_12_H1}
			&\| e_3^1 \|_{H^1} \leq C(M_2) \| e^{is \Delta} \vini - \vini \|_{H^1} \leq C(M_2) \sqrt{\tau} \| \vini \|_{H^2} \leq C(M_2) \sqrt{\tau}, \\
			&\| e_3^2 \|_{H^1} \leq C(M_2) \| \vini - P_N \vini \|_{H^1} \leq C(M_2) h \| \vini \|_{H^2} \leq C(M_2) h,
		\end{aligned}\\
		&\begin{aligned}\label{e3_34_H1}
			&\| e_3^3 \|_{H^1} \lesssim \sqrt{\tau} \| P_N B(P_N\vini) \|_{H^2} \leq \sqrt{\tau} \| B(P_N\vini) \|_{H^2} \leq C(M_2) \sqrt{\tau}, \\
			&\| e_3^4 \|_{H^1} \lesssim h \| B(P_N\vini) \|_{H^2} \leq C(M_2) h,
		\end{aligned}
	\end{align}
	which yields immediately
	\begin{equation}\label{e3_H1}
		\| e_3 \|_{H^1} \leq C(M_2) \tau \left( \sqrt{\tau} + h \right).
	\end{equation}
	Combining \cref{e1_H1,e2_H1,e3_H1}, we obtain \cref{localerror_H^1_1}, which completes the proof.
\end{proof}

\subsection{$ l^\infty $-conditional $ L^2 $- and $ H^1 $-stability}
\begin{proposition}[$ l^\infty $-conditional stability]\label{prop:stability of Phi_AB_sigma>1/2}
	Let $ 0 < \tau <1 $ and $ v, w \in X_N $ such that $ \| v \|_{l^\infty} \leq M $, $ \| w \|_{l^\infty} \leq M $ and $ \| v \|_{H^2} \leq M_1 $. When $ \sigma \geq 1/2 $, we have
	\begin{align}
		&\|  \Phi^\tau(v) -  \Phi^\tau(w) \|_{L^2} \leq  (1 + C_1(M) \tau) \| v - w \|_{L^2}, \label{eq:L2stability_sigma>1/2}\\		
		&\|  \Phi^\tau(v) -  \Phi^\tau(w) \|_{H^1} \leq (1 + C_2(M, M_1, \| V \|_{W^{1, \infty}}) \tau) \| v - w \|_{H^1},\label{eq:H1stability_sigma>1/2}
	\end{align}
\end{proposition}

\begin{proof}
	The $ L^2 $-stability \cref{eq:L2stability_sigma>1/2} can be obtained from \cref{eq:stability_L2_proof} by using \cref{lem:diff_phi_B_2_sigma>1/2} instead of \cref{lem:diff_Phi_B_2_sigma<1/2}. In the following, we show the $ H^1 $-stability \cref{eq:H1stability_sigma>1/2}. By \cref{eq:phi_B_def} and the isometry property of $ e^{i \tau \Delta} $, \cref{eq:H1stability_sigma>1/2} reduces to
	\begin{equation}
		\| I_N \Phi_B^\tau (v) - I_N \Phi_B^\tau (w) \|_{H^1} \leq (1 + C(M, M_1, \| V \|_{W^{1, \infty}}) \tau) \| v - w \|_{H^1}. \label{H1_stability_sigma>1/2}
	\end{equation}
	The proof is based on the following well-known equivalence relation (see, e.g., Lemma 3.2 in \cite{bao2014}): {with $\dxp$ defined in \cref{eq:dxp_def}, }
	\begin{equation}\label{eq:H1_I_N}
		\| \dxp \phi \|_{l^2} \leq \| \nabla I_N \phi \|_{L^2} \leq \frac{\pi}{2} \| \dxp \phi \|_{l^2}, \quad \phi \in X_N, 
	\end{equation}
	which implies, 
	\begin{align}\label{reduce_to_dxp}
		\| I_N \Phi_B^\tau (v) - I_N \Phi_B^\tau (w) \|_{H^1}
		&\leq \| v - w \|_{H^1} + \| I_N (\Phi_B^\tau (v) - v) - I_N (\Phi_B^\tau (w) - w) \|_{H^1} \notag\\
		&\leq \| v - w \|_{H^1} + C \| \dxp (\Phi_B^\tau (v) - v) - \dxp (\Phi_B^\tau (w) - w) \|_{l^2}. 
	\end{align}
	We define
	\begin{align*}
		&v_{j}^\theta = (1-\theta) v_j + \theta v_{j+1}, \quad w_{j}^\theta = (1-\theta) w_j + \theta w_{j+1}, \\
		&V_j^\theta = (1-\theta)V(x_j) + \theta V(x_{j+1}), \quad 0 \leq \theta \leq 1, \quad j = 0, \cdots, N-1.
	\end{align*}
	By some elementary calculations, recalling \cref{eq:G_def} and $ f\p(|z|^2)|z|^2 = \sigma f(|z|^2) $, one gets, for $ 0 \leq j \leq N-1 $,
	\begin{align}\label{dxpv}
			&\dxp (\Phi_B^\tau (v) - v)_j = \dxp \left( v (\expit{v} - 1) \right)_j \notag\\
			&= \frac{1}{h}\int_0^1 \frac{\rmd}{\rmd \theta} \left( v_j^\theta (\expitVjtheta{v_j^\theta} - 1) \right) \rmd \theta \notag\\
			&= \int_0^1 \left[\dxp v_j (\expitVjtheta{v_j^\theta} - 1) - i \tau e^{- i \tau V_j^\theta} v_j^\theta \dxp V(x_j) e^{- i \tau f(|v_j^\theta|^2)} \right. \notag\\
			&\qquad \quad \left. - i \tau e^{- i \tau V_j^\theta} (\sigma f(|v_j^\theta|^2) \dxp v_j + \G{v_j^\theta} \dxp \overline{v_j} )e^{- i \tau f(|v_j^\theta|^2)} \right] \rmd \theta. 
	\end{align}
	Similarly, for $ 0 \leq j \leq N-1 $,
	\begin{align}\label{dxpw}
			&\dxp (\Phi_B^\tau (w) - w)_j \notag\\
			&= \int_0^1 \left[\dxp w_j (\expitVjtheta{w_j^\theta} - 1) - i \tau e^{- i \tau V_j^\theta} w_j^\theta \dxp V(x_j) e^{- i \tau f(|w_j^\theta|^2)} \right. \notag\\
			&\qquad \quad \left. - i \tau e^{- i \tau V_j^\theta} (\sigma f(|w_j^\theta|^2) \dxp w_j + \G{w_j^\theta} \dxp \overline{w_j} )e^{- i \tau f(|w_j^\theta|^2)} \right] \rmd \theta.
	\end{align}
	We define the function $ e \in Y_N $ with
	\begin{equation*}
		e_j = v_j - w_j, \quad j = 0, \cdots, N.
	\end{equation*}
	Subtracting \cref{dxpw} from \cref{dxpv}, for $ 0 \leq j \leq N-1 $, we have
	\begin{align}\label{dxp_decomp}
			&\left| \dxp (\Phi_B^\tau (v) - v)_j - \dxp (\Phi_B^\tau (w) - w)_j \right| \notag\\
			&\leq \int_0^1 \left [ \left| \dxp v_j (\expitVjtheta{v_j^\theta} - 1) - \dxp w_j (\expitVjtheta{w_j^\theta} - 1) \right| \right. \notag\\
			&\quad + \tau \left| \dxp V(x_j) \right| \left| v_j^\theta  e^{- i \tau f(|v_j^\theta|^2)} - w_j^\theta e^{- i \tau f(|w_j^\theta|^2)} \right| \notag\\
			&\quad + \sigma \tau \left| \dxp v_j f(|v_j^\theta|^2) e^{- i \tau f(|v_j^\theta|^2)} - \dxp w_j f(|w_j^\theta|^2)  e^{- i \tau f(|w_j^\theta|^2)} \right| \notag\\
			&\quad \left.+ \tau \left| \dxp \overline{v_j} \G{v_j^\theta} e^{- i \tau f(|v_j^\theta|^2)} - \dxp \overline{w_j} \G{w_j^\theta} e^{- i \tau f(|w_j^\theta|^2)} \right| \right] \rmd \theta \notag\\
			&=: \int_0^1 \left (J^1_j + J^2_j + J^3_j + J^4_j \right ) \rmd \theta.
	\end{align}
	For $ J^1_j $, by \cref{eq:diff_f}, one gets
	\begin{align}\label{J1}
			J^1_j
			&\leq \left| \dxp v_j \right| \left| \expitVjtheta{v_j^\theta} - \expitVjtheta{w_j^\theta} \right| \notag\\
			&\quad + \left| \dxp v_j - \dxp w_j \right| \left| \expitVjtheta{w_j^\theta} - 1 \right| \notag\\
			&\leq \tau \left| \dxp v_j \right| \left| f(|v_j^\theta|^2) - f(|w_j^\theta|^2) \right| + \tau \left| V_j^\theta + f(|w_j^\theta|^2) \right| \left| \dxp v_j - \dxp w_j \right| \notag\\
			&\leq  C(M) \tau \left| \dxp v_j \right| (|e_j| + |e_{j+1}|) + C(M, \| V \|_{L^\infty}) \tau |\dxp e_j|.
	\end{align}
	For $ J^2_j $, recalling \cref{eq:phi_B_2_def}, by \cref{lem:diff_phi_B_2_sigma>1/2} and $ 0 < \tau < 1 $, one gets
	\begin{align}\label{J2}
			J^2_j
			&= \tau \left| \dxp V(x_j) \right| \left | \Phi_{B_2}^\tau(v_j^\theta) - \Phi_{B_2}^\tau(w^\theta_j) \right | \notag \\
			&\leq \tau \left| \dxp V(x_j) \right| (1+C(M)\tau) (|e_j| + |e_{j+1}|) \notag \\
			&\leq C(M) \tau \left| \dxp V(x_j) \right| (|e_j| + |e_{j+1}|).
	\end{align}
	For $ J^3_j $, by \cref{eq:diff_f} and $ 0 < \tau < 1 $, one gets
	\begin{align}\label{J3}
			J^3_j
			&\lesssim \tau \left| \dxp v_j \right| \left| f(|v_j^\theta|^2) e^{- i \tau f(|v_j^\theta|^2)} - f(|w_j^\theta|^2)  e^{- i \tau f(|w_j^\theta|^2)} \right| \notag\\
			&\quad + \tau \left| \dxp v_j - \dxp w_j \right| \left| f(|w_j^\theta|^2)  e^{- i \tau f(|w_j^\theta|^2)} \right| \notag\\
			&\leq \tau \left| \dxp v_j \right| \left ( \left| f(|v_j^\theta|^2) - f(|w_j^\theta|^2) \right| + | f(|w_j^\theta|^2)| \left| e^{- i \tau f(|v_j^\theta|^2)} - e^{- i \tau f(|w_j^\theta|^2)} \right| \right ) \notag\\
			&\quad +\tau C(M) \left| \dxp v_j - \dxp w_j \right| \notag\\
			&\leq \tau \left| \dxp v_j \right| C(M)(1+\tau) |v_j^\theta - w_j^\theta| + \tau C(M) \left| \dxp v_j - \dxp w_j \right| \notag\\
			&\leq C(M) \tau \left| \dxp v_j \right| (|e_j| + |e_{j+1}|) + C(M) \tau \left| \dxp e_j \right|.
	\end{align}
	Similar to \cref{J3}, using \cref{eq:diff_G} instead of \cref{eq:diff_f}, one gets, for $ J^4_j $, 
	\begin{equation}\label{J4}
		J_j^4 \leq C(M) \tau \left| \dxp v_j \right| (|e_j| + |e_{j+1}|) + C(M) \tau \left| \dxp e_j \right|. 
	\end{equation}
	Plugging \cref{J1,J2,J3,J4} into \cref{dxp_decomp}, we have
	\begin{align*}
		 &\left| \dxp (\Phi_B^\tau (v) - v)_j - \dxp (\Phi_B^\tau (w) - w)_j \right| \\
		 &\leq C(M) \tau \left(\left| \dxp V(x_j) \right| + \left| \dxp v_j \right| \right) (|e_j| + |e_{j+1}|) + C(M, \| V \|_{L^\infty}) \tau \left| \dxp e_j \right|,
	\end{align*}
	which yields
	\begin{align}\label{est_dxp}
			&\| \dxp (\Phi_B^\tau (v) - v) - \dxp (\Phi_B^\tau (w) - w) \|_{l^2}^2 \notag\\
			&= h\sum_{j=0}^{N-1} \left| \dxp (\Phi_B^\tau (v) - v)_j - \dxp (\Phi_B^\tau (w) - w)_j \right|^2 \notag\\
			&\leq C(M) \tau^2 h \sum_{j=0}^{N-1} \left(\left| \dxp V(x_j) \right|^2 + \left| \dxp v_j \right|^2 \right) (|e_j|^2 + |e_{j+1}|^2) \notag\\
			&\quad + C(M, \| V \|_{L^\infty})\tau^2 h \sum_{j=0}^{N-1}  \left| \dxp e_j \right|^2 \notag\\
			&\leq C(M) \tau^2 \left ( \| \nabla  V \|_{L^\infty}^2 \| e \|_{l^2}^2 + h \sum_{j=0}^{N-1} \left| \dxp v_j \right|^2 (|e_j|^2 + |e_{j+1}|^2) \right) \notag\\
			&\quad +  C(M, \| V \|_{L^\infty})\tau^2 \| \dxp e \|_{l^2}^2.
	\end{align}
	When $ d=1 $, one has $ |\dxp v_j| \leq \| \nabla v \|_{L^\infty} \leq C(M_1) $, which yields directly that
	\begin{equation}\label{1D_est}
		h \sum_{j=0}^{N-1} \left| \dxp v_j \right|^2 (|e_j|^2 + |e_{j+1}|^2) \leq C(M_1) \| e \|_{l^2}^2.
	\end{equation}
	However, \cref{1D_est} cannot be directly generalized to 2D and 3D without assuming higher regularity on $ v $. Here, we present an alternative approach that can be generalized to 2D and 3D (see also \cref{rem:3D}). Using the discrete Gagliardo-Nirenberg inequality ({(2.4) in \cite{FD} or (3.3) in \cite{bao2013}}) and the discrete Poincar\'e inequality ((3.3) in \cite{bao2013}), we have
	\begin{equation}\label{discrete G-N}
		\| \phi \|_{l^4} \lesssim \| \phi \|_{l^2}^\frac{3}{4} \| \dxp \phi \|_{l^2}^\frac{1}{4} \lesssim \| \dxp \phi\|_{l^2}, \quad \phi \in Y_N,
	\end{equation}
	which implies, by first applying H\"older's inequality in \cref{1D_est},
	\begin{equation}\label{S1_Holder}
		\begin{aligned}
			h \sum_{j=0}^{N-1} \left| \dxp v_j \right|^2 (|e_j|^2 + |e_{j+1}|^2) \lesssim \| \dxp v \|_{l^4}^2 \| e \|_{l^4}^2 \lesssim \| \dxp v \|_{l^4}^2 \| \dxp e \|_{l^2}^2.
		\end{aligned}
	\end{equation}
	Using the following discrete version of the Sobolev embedding $ H^2 \hookrightarrow W^{1, 4} $ ({see (3.3) in \cite{bao2013} and also the appendix})
	\begin{equation}\label{discrete_embedding}
		\| \dxp \phi \|_{l^4} \lesssim \| \phi \|_{H^2}, \quad \phi \in X_N,
	\end{equation}
	we get $ \| \dxp v \|_{l^4} \lesssim \| v \|_{H^2} \leq M_1 $, which plugged into \cref{S1_Holder} yields from \cref{est_dxp}
	\begin{align}\label{est_dxp_conclusion}
		&\| \dxp (\Phi_B^\tau (v) - v) - \dxp (\Phi_B^\tau (w) - w) \|_{l^2}^2 \notag \\
		&\leq C(M, M_1, \| V \|_{W^{1, \infty}}) \left (\| e \|_{l^2}^2 + \| \dxp e \|_{l^2}^2 \right ).
	\end{align}
	From \cref{est_dxp_conclusion}, using the discrete Poincar\'e inequality and \cref{eq:I_N_L2,eq:H1_I_N}, we have
	\begin{align}
			\| \dxp (\Phi_B^\tau (v) - v) - \dxp (\Phi_B^\tau (w) - w) \|_{l^2}^2
			&\leq C(M, M_1, \| V \|_{W^{1, \infty}}) \| \dxp e \|_{l^2}^2 \notag\\
			&\leq  C(M, M_1, \| V \|_{W^{1, \infty}}) \| \nabla I_N e \|_{L^2}^2,
	\end{align}
	which plugged into \cref{reduce_to_dxp} yields \cref{H1_stability_sigma>1/2}, and completes the proof.
\end{proof}

\begin{remark}\label{rem:3D}
	The 2D case follows exactly \cref{discrete G-N,S1_Holder,discrete_embedding}. The proof of \cref{discrete_embedding} in 2D proceeds similarly to our proof in 1D in the appendix by following the proof of (3.3) in \cite{bao2013} with additional attention paid to the boundary terms. The 3D case follows \cref{discrete G-N,S1_Holder,discrete_embedding} with slight modification: using H\"older's inequality with index $ (3/2, 3) $ in \cref{S1_Holder}. Then the discrete version of $ H^1 \hookrightarrow L^6 $ and $ H^2 \hookrightarrow W^{1, 3} $ in 3D are needed. The proof of the first one can be found in \cite{dGN} while the proof of the second one will follow the proof of \cref{discrete_embedding} in 2D, which is the reason why we modify the estimates in 3D. 
\end{remark}

\subsection{Proof of \cref{eq:thm2_part1} in \cref{thm:sigma>1/2H2}}
With \cref{prop:local_error_sigma>1/2,prop:stability of Phi_AB_sigma>1/2}, we are able to obtain \cref{eq:thm2_part1}.
\begin{proof}[Proof of \cref{eq:thm2_part1} in \cref{thm:sigma>1/2H2}]
	Following the proof of \cref{thm:sigma>1/2H2}, we only need to estimate $ e^k = I_N \psi^k - P_N \psi(\cdot, t_k) $ for $ 0 \leq k \leq T/\tau $. We shall first prove the error estimate in $ H^1 $ norm by the standard argument of the mathematical induction. Replacing $ \| \cdot \|_{L^2} $ with $ \| \cdot \|_{H^1} $ in \cref{eq:error_propagation}, one has for $ 0 \leq k \leq T/\tau - 1 $,
	\begin{align}\label{eq:error_propagation_H1}
		\| e^{k+1} \|_{H^1} 
		&\leq  \| \Phi^\tau (I_N \psi^{k}) - \Phi^\tau (P_N \psi(\cdot, t_k)) \|_{H^1} \notag \\
		&\quad + \| \Phi^\tau (P_N \psi(\cdot, t_k)) - P_N \psi(\cdot, t_{k+1}) \|_{H^1}.
	\end{align}
	When $ k = 0 $, by \cref{eq:proj_inter_error_H1}, one gets
	\begin{equation*}
		\| e^0 \|_{H^1} = \| I_N \psi_0 - P_N \psi_0 \|_{H^1} \lesssim h \| \psi_0 \|_{H^2} \leq C(M_2) h, \  \| \psi^0 \|_{l^\infty} \leq \| \psi_0 \|_{L^\infty} \leq 1 + M_2.
	\end{equation*}
	We assume that for $ 0 \leq k \leq m \leq T/\tau-1 $,
	\begin{equation}\label{eq:assumption}
		\| e^k \|_{H^1} \lesssim \tau^\frac{1}{2} + h, \quad \| \psi^k \|_{l^\infty} \leq 1 + M_2.
	\end{equation}
	We shall prove \cref{eq:assumption} for $ m+1 $. From \cref{eq:error_propagation_H1}, using \cref{eq:H1stability_sigma>1/2,localerror_H^1_1}, and noting the assumption \cref{eq:assumption}, we have
	\begin{equation}\label{error_eq}
		\| e^{m+1} \|_{H^1} \leq (1 + C_1 \tau) \| e^m \|_{H^1} + C_2 \tau\left(\tau^\frac{1}{2} + h\right),
	\end{equation}
	where $ C_1 $ and $ C_2 $ are the constants in \cref{eq:H1stability_sigma>1/2,localerror_H^1_1} respectively, which depend exclusively on $ M_2 $ and $ \| V \|_{W^{1, \infty}} $. From \cref{error_eq}, standard discrete Gronwall's inequality yields
	\begin{equation}\label{H^1_m+1}
		\| e^{m+1} \|_{H^1} \leq 2e^{C_0T}C_1\left(\tau^\frac{1}{2} + h\right).
	\end{equation}
	Recalling that $ e^{k} = I_N \psi^k - P_N\psi(t_k) $ and $ \| \psi(\cdot, t_k) \|_{L^\infty} \leq M_2 $, using the inverse inequality $ \| \phi \|_{L^\infty} \lesssim h^{-1/2} \| \phi \|_{L^2} $, $ \phi \in X_N $ {\cite{book_spectral}}, we have
	\begin{equation*}
		\begin{aligned}
			\| \psi^{m+1} \|_{l^\infty}
			&= \| I_N \psi^{m+1} \|_{l^\infty} \leq \| e^{m+1} \|_{l^\infty} + \| P_N \psi(\cdot, t_{m+1}) \|_{l^\infty} \\
			&\leq \| e^{m+1} \|_{l^\infty} + \| \psi(\cdot, t_{m+1}) - P_N \psi(\cdot, t_{m+1}) \|_{l^\infty} + \| \psi(\cdot, t_{m+1}) \|_{l^\infty} \\
			&\leq \| e^{m+1} \|_{l^\infty} +   h^{-\frac{1}{2}}\| \psi(\cdot, t_{m+1}) - P_N \psi(\cdot, t_{m+1}) \|_{L^2} + M_2.
		\end{aligned}
	\end{equation*}
	Hence, for $ \tau \leq \tau_0 $ and $ h \leq h_0 $ with $ \tau_0>0 $ and $ h_0>0 $ depending on $ M_2 $ and $ T $, by Sobolev embedding {$H^1 \hookrightarrow L^\infty$ in 1D}, and \cref{H^1_m+1,eq:proj_inter_error_L2}, we have
	\begin{equation}\label{Linfty_m+1}
		\| \psi^{m+1} \|_{l^\infty} \leq C\| e^{m+1} \|_{H^1} + Ch^{2-1/2} + M_2 \leq 1 + M_2.
	\end{equation}
	Combining \cref{H^1_m+1,Linfty_m+1} proves \cref{eq:assumption} for $ k = m+1 $ and thus for all $ 0 \leq k \leq T/\tau $ by mathematical induction. With the $ l^\infty $-bound of the numerical solution, the $ L^2 $ estimate of $ e^k $ follows the proof of \cref{thm:sigma<1/2} by using \cref{localerror:L^2_2,eq:L2stability_sigma>1/2}, which completes the proof of \cref{eq:thm2_part1}.
\end{proof}
\begin{remark}\label{rem:dgeq2_case_1_thm}
	In 2D and 3D, we no longer have $ H^1 \hookrightarrow L^\infty $. To obtain the $ l^\infty $-bound of $ \psi^{m+1} $ in \cref{Linfty_m+1}, we use the discrete Sobolev inequalities as in \cite{bao2012,bao2013,bao2014}
	\begin{equation*}
		\| v \|_{l^\infty} \leq C|\ln h|\, \| I_N v \|_{H^1}, \quad \| w \|_{l^\infty} \leq Ch^{-1/2} \| I_N w \|_{H^1},
	\end{equation*}
	where $ v $ and $ w $ are 2D and 3D mesh functions with zero at the boundary, respectively, and the interpolation operator $ I_N $ can be defined similarly in 2D and 3D as in 1D. Thus by requiring that the time step size $ \tau $ satisfies the additional assumption \cref{cfl}, we can control the $ l^\infty $-norm of the numerical solution.
\end{remark}

\subsection{Proof of \cref{eq:thm2_part2} in \cref{thm:sigma>1/2H2}}
In the following, we assume that $ 1/2 < \sigma < 1 $, $ V \in H^3(\Omega) $, $ \nabla V \in H_0^1(\Omega) $, $ \psi \in C([0, T]; H^3_\ast(\Omega)) \cap C^1([0, T]; H^1(\Omega)) $ and let
\begin{equation}\label{M3}
	M_3 := \max\left\{ \| \psi \|_{L^\infty([0, T]; H^3)}, \| \psi \|_{L^\infty([0, T]; L^\infty)}, \| V \|_{H^3} \right\}.
\end{equation}
We first show an analogous result of \cref{lem:fractional_error_1}.
\begin{lemma}\label{lem:fractional_error_2}
	Let $ \phi \in X_N $ such that $ \| \phi \|_{H^3} \leq M $ and let $ 0<\tau<1 $ and $ 0<h<1 $. Assume that $ V \in H^3(\Omega) $ and $ \nabla V \in H_0^1(\Omega) $. When $ 1/2 < \sigma < 1 $, we have
	\begin{align}
		&\| (I-e^{i \tau \Delta}) P_N B(\phi) \|_{H^1} \leq C_1(M, \| V \|_{H^3}) \tau^{\sigma}, \label{eq:free_schrodinger_error_2}\\
		&\| I_N B(\phi) - P_N B(\phi) \|_{H^1} \leq C_2(M, \| V \|_{H^3}) h^{2\sigma}. \label{eq:inter_error_2}
	\end{align}
\end{lemma}

\begin{proof}
	Similar to \cref{Vphi_reg_e}, noting that $ V \phi \in H_\ast^{3}(\Omega) $, we have
	\begin{equation}\label{Vphi_e}
		\begin{aligned}
			&\| (I-e^{i \tau \Delta}) (V \phi)  \|_{H^1} \lesssim \tau \| V \|_{H^3}\| \phi \|_{H^3}, \\
			&\| (I_N - P_N) (V \phi) \|_{H^1} \lesssim h^2\| V \|_{H^3}\| \phi \|_{H^3}.
		\end{aligned}
	\end{equation}
	Following \cref{f-fe+fe,f_decomp} with $ \| \cdot \|_{H^1} $ replacing $ \| \cdot \|_{L^2} $ and using \cref{eq:Be_H3}, we have
	\begin{equation}\label{exp_error}
		\| (I-e^{i \tau \Delta}) (f(|\phi|^2)\phi) \|_{H^1} \leq 2\| f(|\phi|^2)\phi - \fe(|\phi|^2)\phi \|_{H^1} + C(M) \frac{\tau}{\vep^{2-2\sigma}}.
	\end{equation}
	By direct calculation, recalling \cref{eq:G_def}, one gets
	\begin{align}\label{eq:nabla_diff_f_fe}
		&\nabla [f(|\phi|^2)\phi - \fe(|\phi|^2)\phi]
		= (f(|\phi|^2) - \fe(|\phi|^2))\nabla \phi \notag \\
		&+ (f\p(|\phi|^2)|\phi|^2 - \fe\p(|\phi|^2)|\phi|^2)\nabla \phi + (G(\phi) - \fe\p(|\phi|^2)\phi^2) \nabla \overline{\phi}. 
	\end{align}
	{Noting that $ f\p(|z|^2)|z|^2 = \fe\p(|z|^2)|z|^2$, $ G(z) = \fe\p(|z|^2)z^2 $ when $ |z|\geq\vep $ and $ f\p(|z|^2)|z|^2 +  \fe\p(|z|^2)|z|^2 + |G(z)| + |\fe\p(|z|^2)z^2| \lesssim \vep^{2\sigma} $ when $ |z| < \vep $, one gets
	\begin{equation*}
		\left|f\p(|z|^2)|z|^2 - \fe\p(|z|^2)|z|^2\right| \lesssim \vep^{2\sigma} \mathbbm{1}_{|z|<\vep}, \quad \left|G(z) - \fe\p(|z|^2)z^2\right| \lesssim \vep^{2\sigma} \mathbbm{1}_{|z|<\vep}, \quad z \in \C, 
	\end{equation*}}
	which together with \cref{eq:f-fe1} applied to \cref{eq:nabla_diff_f_fe} yields 
	\begin{equation}\label{ffeH1est2}
		\| \nabla [f(|\phi|^2)\phi - \fe(|\phi|^2)\phi] \|_{L^2} \lesssim \vep^{2\sigma} \| \nabla \phi \mathbbm{1}_{|\phi|<\vep} \|_{L^2} \leq \vep^{2\sigma} \| \phi \|_{H^1}.
	\end{equation}
	Plugging \cref{ffeH1est2} into \cref{exp_error}, we have
	\begin{equation*}
		\| (I-e^{i \tau \Delta}) (f(|\phi|^2)\phi) \|_{H^1} \leq C(M)\inf_{0<\vep<1} \left(\vep^{2\sigma} + \frac{\tau}{\vep^{2-2\sigma}} \right) \leq C(M)\tau^{\sigma},
	\end{equation*}
	which combined with \cref{Vphi_e} yields \cref{eq:free_schrodinger_error_2}. 
	
	Then we shall show \cref{eq:proj_inter_error_H1}. Following \cref{f_decomp2} with $ \| \cdot \|_{H^1} $ replacing $ \| \cdot \|_{L^2} $, using the standard estimates of $ P_N $, and \cref{ffeH1est2,eq:Be_H3}, one gets
	\begin{align}\label{IP_decomp_reg}
			&\| (I_N - P_N) (f(|\phi|^2)\phi) \|_{H^1} \notag\\
			&\leq \| I_N(f(|\phi|^2)\phi - \fe(|\phi|^2)\phi) \|_{H^1} + \vep^{2\sigma} \| \phi \|_{H^1} + h^2 \frac{C(M)}{\vep^{2-2\sigma}}.
	\end{align}
	Using \cref{eq:H1_I_N}, one gets
	\begin{equation}\label{nabla_to_dxp_reg}
		\| \nabla I_N(f(|\phi|^2)\phi - \fe(|\phi|^2)\phi) \|_{L^2} \lesssim \| \dxp (f(|\phi|^2)\phi - \fe(|\phi|^2)\phi) \|_{l^2}.
	\end{equation}
	Let $ \phi_j^\theta = (1-\theta)\phi_j + \theta \phi_{j+1} $ for $ 0 \leq \theta \leq 1 $ and $ 0 \leq j \leq N-1 $, direct calculation gives
	\begin{align}
		&\dxp \left (f(|\phi_j|^2)\phi_j - \fe(|\phi_j|^2)\phi_j \right )
		= \frac{1}{h}\int_0^1 \frac{\rmd}{\rmd \theta} \left (f(|\phi_j^\theta|^2)\phi_j^\theta - \fe(|\phi_j^\theta|^2)\phi_j^\theta \right ) \rmd \theta \notag \\
		&= \int_0^1 \left [ \left ( \left (f(|\phi_j^\theta|^2) + f\p(|\phi_j^\theta|^2)|\phi_j^\theta|^2 \right ) - \left (\fe(|\phi_j^\theta|^2) + \fe\p(|\phi_j^\theta|^2)|\phi_j^\theta|^2 \right ) \right ) \dxp \phi_j \right. \notag \\
		&\qquad\quad \left. + (G(\phi_j^\theta) - \fe\p(|\phi_j^\theta|^2)(\phi_j^\theta)^2) \dxp \overline{\phi_j} \right ] \rmd \theta,
	\end{align}
	which implies, {similar to the way we obtain \cref{ffeH1est2} from \cref{eq:nabla_diff_f_fe}},
	\begin{equation}\label{dxp_est_reg}
		\left | \dxp (f(|\phi_j|^2)\phi_j - \fe(|\phi_j|^2)\phi_j) \right | \lesssim \vep^{2\sigma} |\dxp \phi_j|.
	\end{equation}
	From \cref{nabla_to_dxp_reg}, using \cref{dxp_est_reg} and recalling \cref{eq:H1_I_N} and $ \phi \in X_N $, we obtain
	\begin{equation}
		\begin{aligned}
			\| \nabla I_N(f(|\phi|^2)\phi - \fe(|\phi|^2)\phi) \|_{L^2} \lesssim \vep^{2\sigma} \| \dxp \phi \|_{l^2} \leq \vep^{2\sigma} \| \nabla \phi \|_{L^2},
		\end{aligned}
	\end{equation}
	which plugged into \cref{IP_decomp_reg} yields
	\begin{equation*}
		\| (I_N-P_N) (f(|\phi|^2)\phi) \|_{H^1} \leq C(M)\inf_{0<\vep<1} \left(\vep^{2\sigma} + \frac{h^2}{\vep^{2-2\sigma}} \right) \leq C(M) h^{2\sigma},
	\end{equation*}
	which combined with \cref{Vphi_e} yields \cref{eq:inter_error_2} and completes the proof.
\end{proof}

\begin{proposition}[local truncation error]\label{prop:local_error_H3}
	Assume that $ V \in H^3(\Omega) $, $ \nabla V \in H_0^1(\Omega) $, $ \psi \in C([0, T]; H^3_\ast(\Omega)) \cap C^1([0, T]; H^1(\Omega)) $ and $ 1/2 < \sigma < 1  $. For $ 0 \leq k \leq T/\tau -1 $, we have
	\begin{equation*}
		\| P_N \psi(\cdot, t_{k+1}) - \Phi^\tau( P_N \psi(\cdot, t_k) ) \|_{H^1(\Omega)} \leq C(M_3) \tau \left(\tau^\sigma + h^{2\sigma} \right).
	\end{equation*}
\end{proposition}

\begin{proof}
	Following the proof of \cref{prop:local_error_sigma>1/2}, we only need to modify the estimate \cref{e3_12_H1,e3_34_H1}, which can be easily done by using the assumption $ \psi \in C([0, T]; H^3) $, \cref{lem:fractional_error_2}, and the standard estimates of the operators $ I_N - P_N $, $ I - P_N $ and $ I - e^{i \tau \Delta} $.
\end{proof}

\begin{proof}[Proof of \cref{eq:thm2_part2} in \cref{thm:sigma>1/2H2}]
	Using \cref{prop:local_error_H3} and \cref{eq:H1stability_sigma>1/2} in \cref{eq:error_propagation_H1}, and noting the $ l^\infty $-bound of the numerical solution in \cref{eq:thm2_part1}, then \cref{eq:thm2_part2} follows from the discrete Gronwall's inequality immediately.
\end{proof}


\section{Numerical results}
In this section, we present some numerical examples for the NLSE \cref{NLSE} with $ 0 < \sigma < 1 $ in 1D to confirm our error estimates. Since we are mainly interested in the semi-smooth nonlinearity, we choose $ V(x) \equiv 0 $, and consider the following two initial set-ups:
\begin{description}
	\item[Type I] We consider the smooth initial datum
	\begin{equation}\label{typeI_ini}
		\psi_0(x) = x e^{-\frac{x^2}{2}}, \quad x \in \Omega=(-16, 16).
	\end{equation}
	
	\item[Type II] We consider the initial datum in $ H^2(\Omega) $ as in \cite{LRI_sinum}
	\begin{equation}\label{typeII_ini}
		\begin{aligned}
			&\psi_0 = \frac{\phi^{(1)}}{\| \phi^{(1)} \|_{L^2}}, \quad \phi^{(1)}(x) = \sum_{l \in \mathcal{T}_N} \widetilde \phi_l^{(1)} \sin(\mu_l(x-a)), \quad x \in \Omega = (-1, 1), \\
			&\widetilde \phi_l^{(1)} = \frac{\widetilde \phi_l}{|\mu_l|^{2.5}}, \quad \widetilde \phi_l = \left\{
			\begin{aligned}
				&\text{rand}(-1, 1) + i \  \text{rand}(-1, 1), && l \text{ even}, \\
				&0, && l \text{ odd},
			\end{aligned}
			\right. \quad l \in \mathcal{T}_N. 
		\end{aligned}
	\end{equation}
	where $ \text{rand}(-1, 1) $ returns a uniformly distributed random number between $ -1 $ and $ 1 $.
\end{description}
Note that both Types I and II initial data are chosen as odd functions to demonstrate the influence of the semi-smoothness of $ f $ at the origin since with an odd initial datum, the exact solution satisfies $ \psi(0, t) \equiv 0 $ for all $ t \geq 0 $. 
{In \cref{fig:type_1_ini} (a), we plot the density of the wave functions at $t=1$ with different $\sigma = 0.1, 0.25, 0.5, 1$ and $\beta = -10$ for the Type I initial datum. We observe that the solution of the smooth case ($\sigma = 1$) lies between the solution of the case $\sigma = 0.1$ and $\sigma = 0.5$, and is close to the solution of the case $\sigma = 0.25$. In \cref{fig:type_1_ini} (b), we plot the relative errors of the energy divided by $\tau$ up to $t=8$ for $\sigma = 0.1$ and different $\tau = 0.05, 0.01, 0.002$. We see that the relative error of the energy is at $O(\tau)$ with fixed mesh size $h$. }
\begin{figure}[htbp]
	\centering
	{\includegraphics[width=0.75\textwidth]{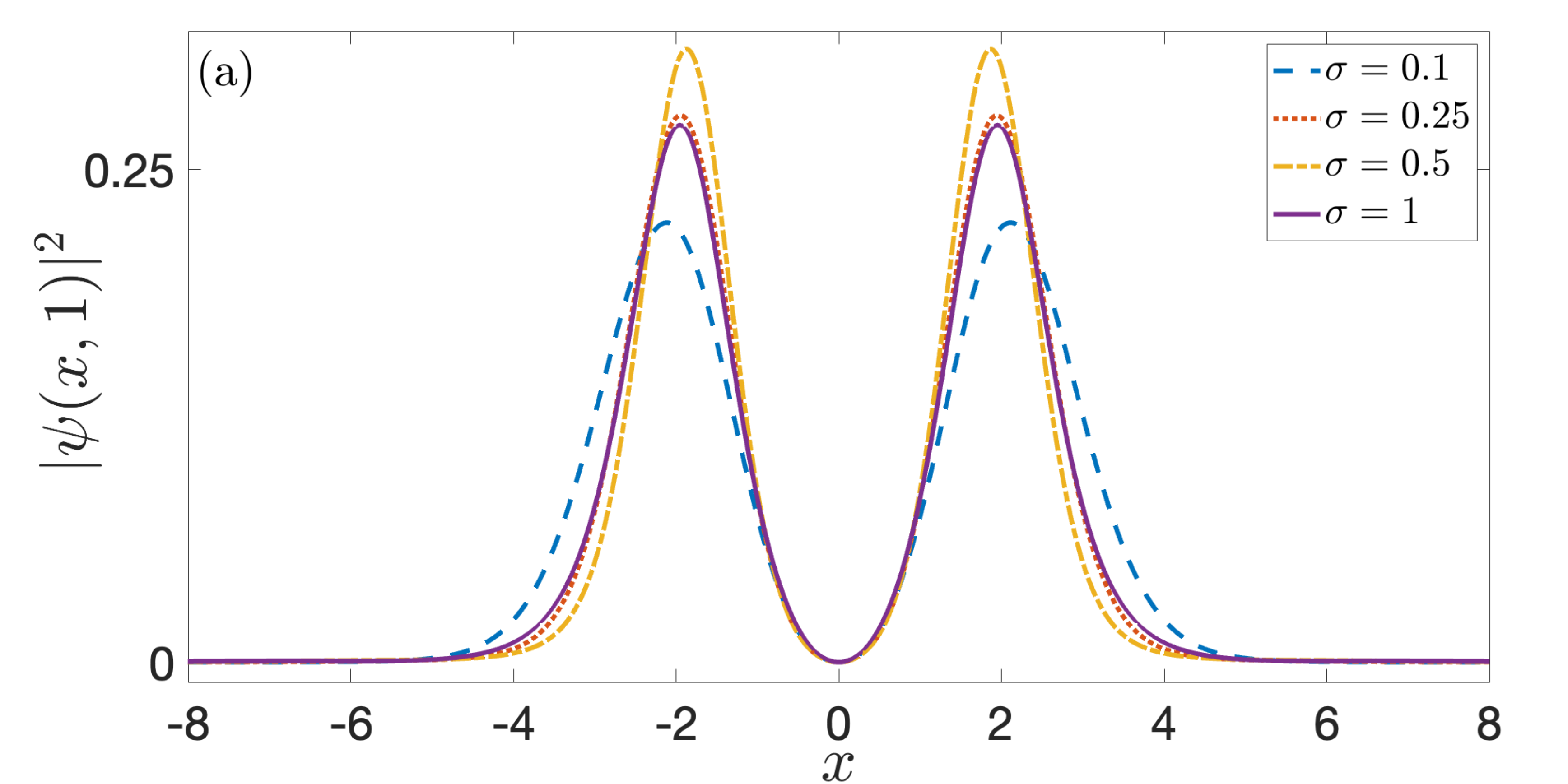}}\\
	{\includegraphics[width=0.75\textwidth]{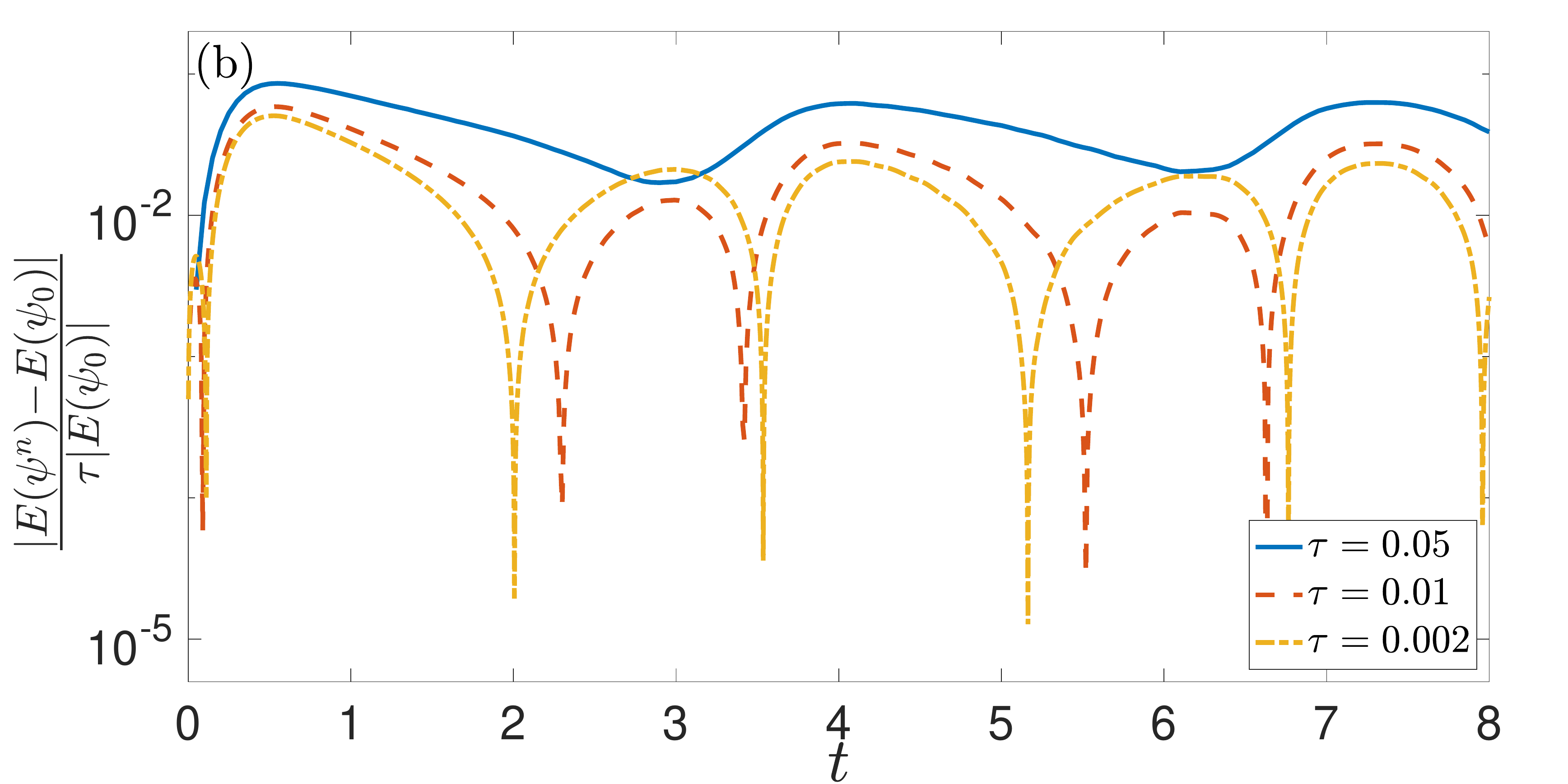}}
	\caption{(a) density $ |\psi(x, 1)|^2 $ with different $ \sigma >0 $ and (b) relative errors of the energy divided by $\tau$ up to $t=8$ with $\sigma = 0.1$ for the Type I initial datum \cref{typeI_ini} with $\beta = -10$. }
	\label{fig:type_1_ini}
\end{figure}

In the following, we shall test the errors of the TSSP in $L^2$- and $H^1$-norms. We fix $ d=1 $, $ T=1 $ and $\beta = -1$. The NLSE \cref{NLSE} is then solved by the TSSP method on the domain $ \Omega $ with Type I and Type II initial setups for different $ \sigma>0 $. The `exact' solution is obtained numerically by the Strang splitting sine pseudospectral method with a very fine mesh size $ h_e = 2^{-9} $ and a small time step size $ \tau_e = 10^{-6} $. In our numerical experiments below, when testing the temporal convergence, we always fix the mesh size $ h = h_e $. To quantify the error, we introduce the following error functions:
\begin{align*}
	e_{L^2}^k = \| \psi(\cdot, t_k) - I_N \psi^k  \|_{L^2}, \quad e^k_{H^1} = \| \psi(\cdot, t_k) - I_N \psi^k \|_{H^1}, \quad 0 \leq k \leq n: = T/\tau.
\end{align*}

\cref{fig:exmp1_L2_multi_sigma} exhibits the temporal and spatial errors in $ L^2 $-norm of the TSSP \cref{full_discretization_scheme} for the NLSE \cref{NLSE} with Type I initial datum and different $ 0<\sigma\leq 1/2 $. \cref{fig:exmp1_L2_multi_sigma} (a) shows that the temporal convergence is first order in $ L^2 $-norm for all the four $ \sigma $, and \cref{fig:exmp1_L2_multi_sigma} (b) shows the spatial convergence is almost third order in $ L^2 $-norm, which is also increasing with $ \sigma $. These results are better than our error estimates in \cref{thm:sigma<1/2} and suggest that first order temporal convergence in $ L^2 $-norm may hold for any $ \sigma > 0 $ and the spatial convergence may be of higher order. Similar results are also observed in our numerical experiments in 2D. However, we remark that it is impossible to obtain the optimal temporal convergence rates and the higher order spatial convergence rates by simply improving the local error estimates in \cref{prop:local_error_sigma<1/2}, indicating that there must exist error cancellation between different steps, which require new techniques and in-depth analysis to handle. {Also, we can observe similar temporal errors as in \cref{fig:exmp1_L2_multi_sigma} (a) for the Type II initial datum. }

	\begin{figure}[htbp]
		\centering
		{\includegraphics[width=0.475\textwidth]{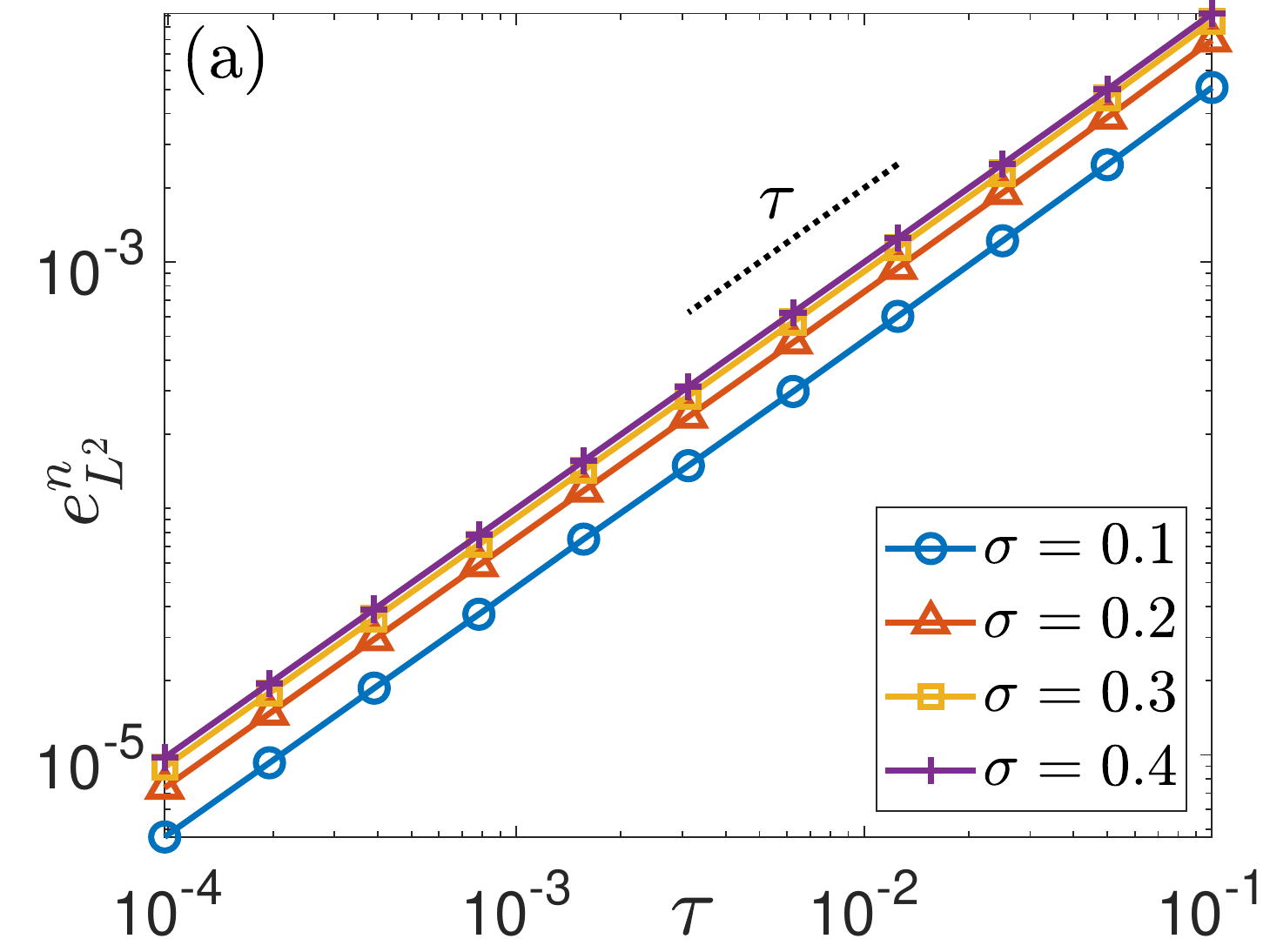}}
		{\includegraphics[width=0.475\textwidth]{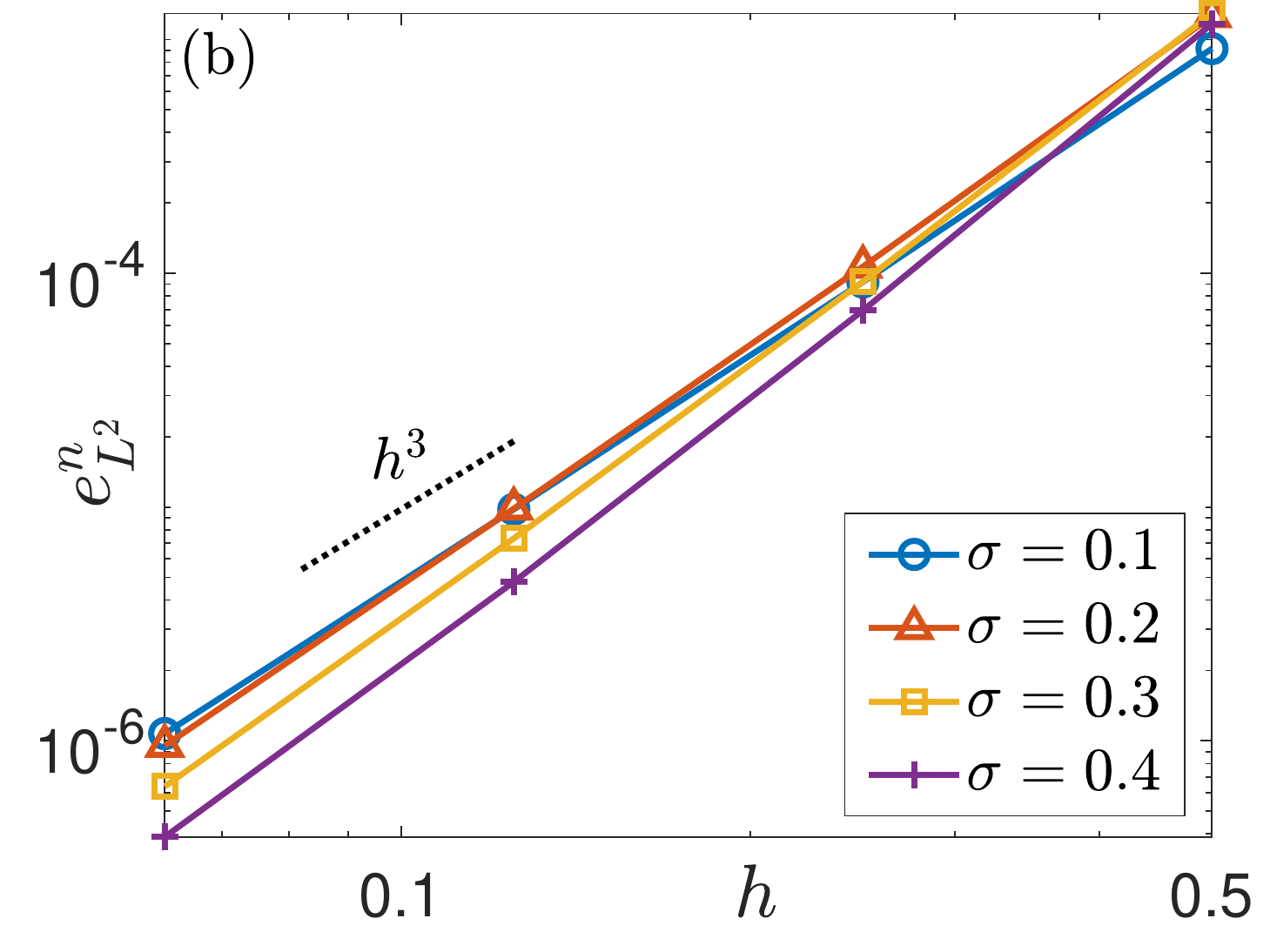}}
		\caption{Temporal errors (a) and spatial errors (b) in $ L^2 $-norm for $ \sigma=0.1, 0.2, 0.3, 0.4 $ with Type I initial datum \cref{typeI_ini}. }
		\label{fig:exmp1_L2_multi_sigma}
	\end{figure}

\cref{fig:exmp1_conv_dt} plots the temporal and spatial errors in $ L^2 $- and $ H^1 $-norm of the TSSP \cref{full_discretization_scheme} for the NLSE \cref{NLSE} with Type II $ H^2 $ initial datum and fixed $ \sigma = 0.5 $. \cref{fig:exmp1_conv_dt} (a) shows that the temporal convergence is first order in $ L^2 $-norm and half order in $ H^1 $-norm, and \cref{fig:exmp1_conv_dt} (b) shows the spatial convergence is second order in $ L^2 $-norm and first order in $ H^1 $-norm. These results correspond with our error estimates \cref{eq:thm2_part1} in \cref{thm:sigma>1/2H2} very well.

	\begin{figure}[htbp]
		\centering
		{\includegraphics[width=0.475\textwidth]{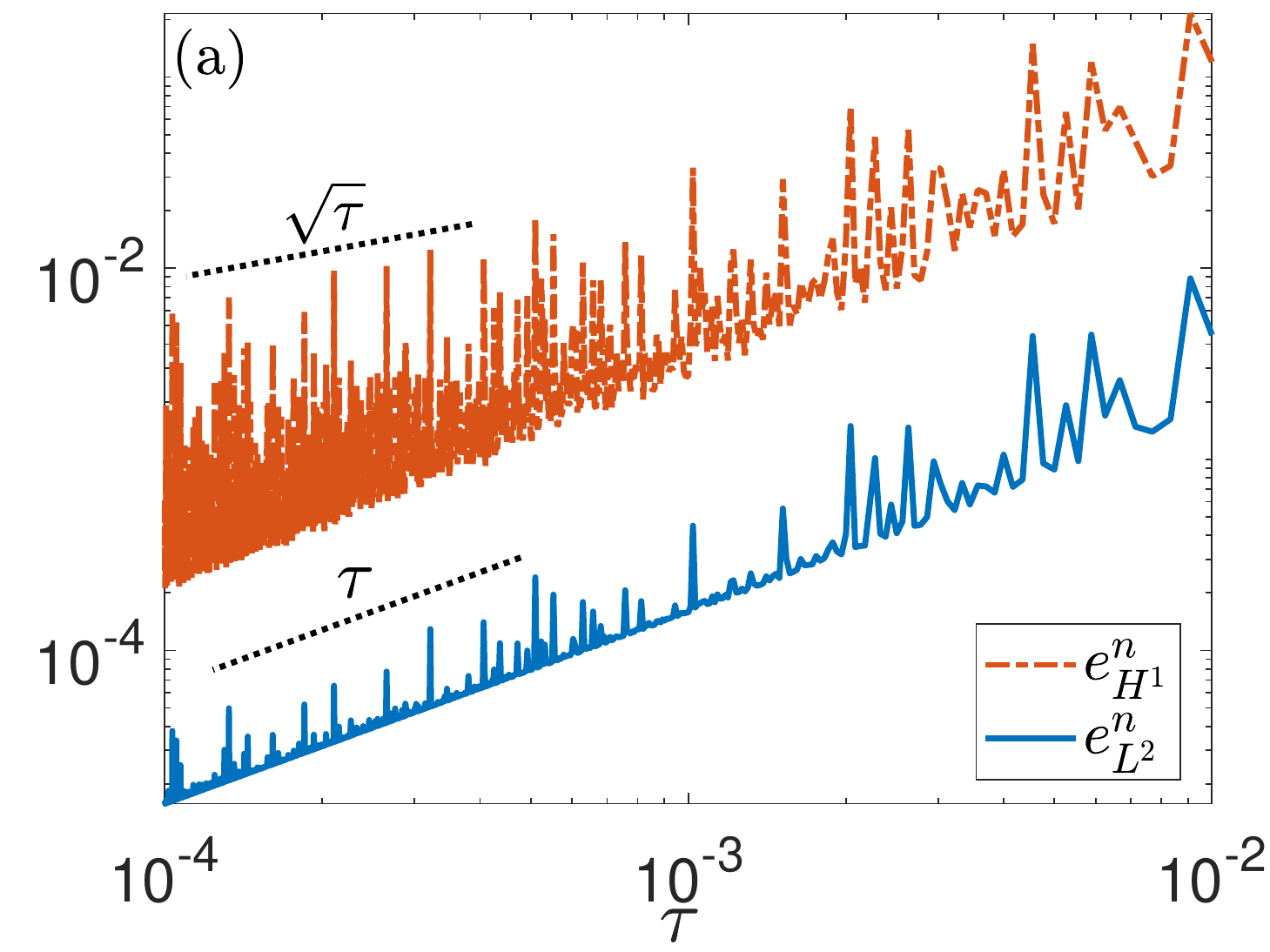}}
		{\includegraphics[width=0.475\textwidth]{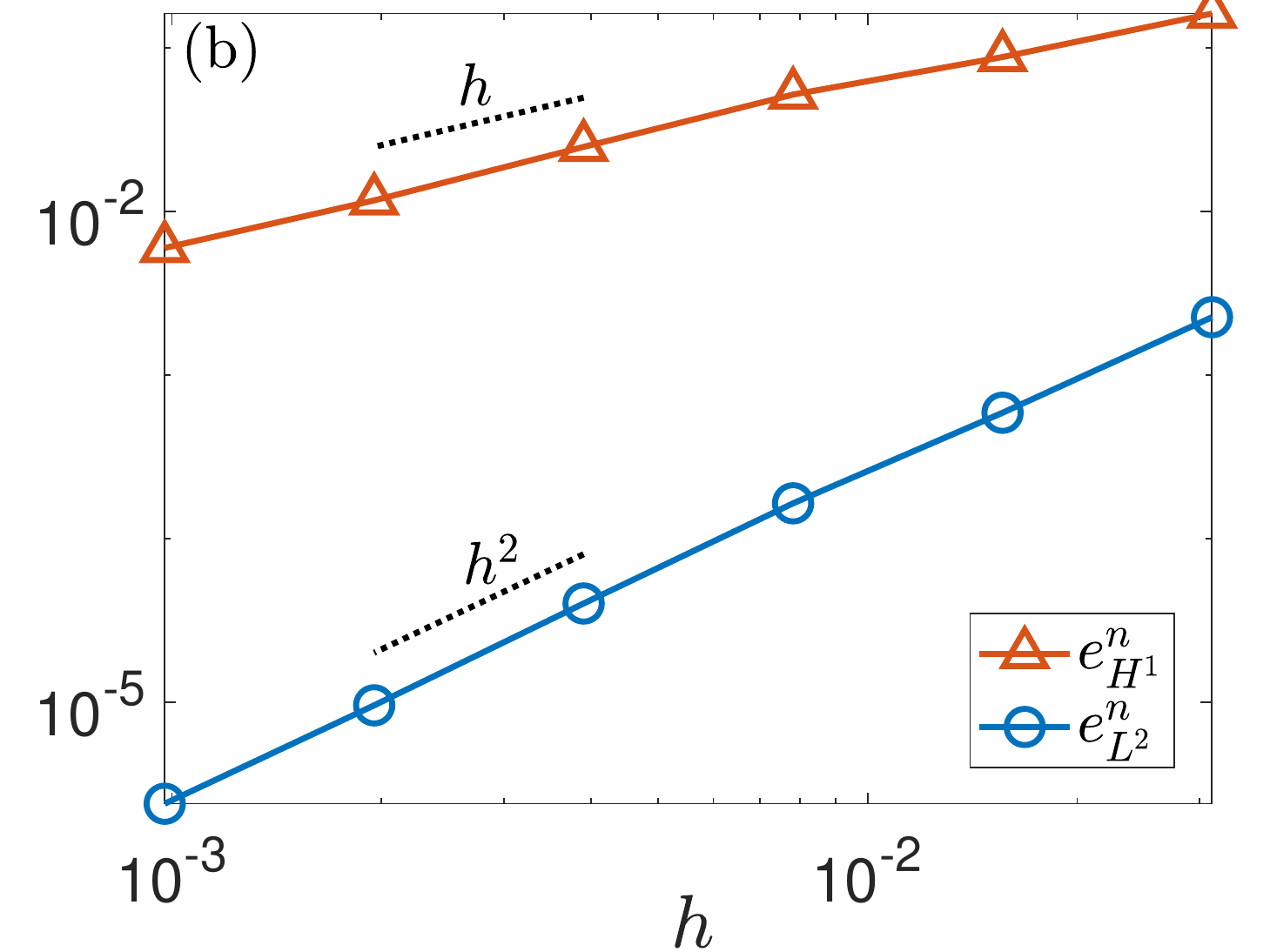}}
		\caption{Temporal errors (a) and spatial errors (b) in $ L^2 $-norm and $ H^1 $-norm for $ \sigma=0.5 $ with Type II initial data \cref{typeII_ini}. }
		\label{fig:exmp1_conv_dt}
	\end{figure}

\cref{fig:exmp3_H1_multi_sigma} displays the temporal and spatial errors in $ H^1 $-norm of the TSSP \cref{full_discretization_scheme} for the NLSE \cref{NLSE} with Type I smooth initial datum and different $ 0<\sigma<1 $. \cref{fig:exmp3_H1_multi_sigma} (a) shows that the temporal convergence in $ H^1 $-norm increases from half order to first order as $ \sigma $ increase from $ 0 $ to $ 1/2 $ and remains first order when $ \sigma \geq 1/2 $. \cref{fig:exmp3_H1_multi_sigma} (b) shows the spatial convergence is almost 2.5 order in $ H^1 $-norm and is increasing with $ \sigma $. Similar to the observation of \cref{fig:exmp1_L2_multi_sigma}, these results are better than our error estimates \cref{eq:thm2_part2} in \cref{thm:sigma>1/2H2} and suggest that first order temporal convergence in $ H^1 $-norm may hold for any $ \sigma \geq 1/2 $. We would like to comment that the order reduction in $ H^1 $-norm for $ 0 < \sigma < 1/2 $ is indeed resulted from the semi-smoothness of the nonlinearity instead of the regularity of the exact solution. Actually, we numerically checked that with the Type I smooth initial datum, the exact solution is roughly in $ H^{3.5+2\sigma} $.

	\begin{figure}[htbp]
	\centering
	{\includegraphics[width=0.475\textwidth]{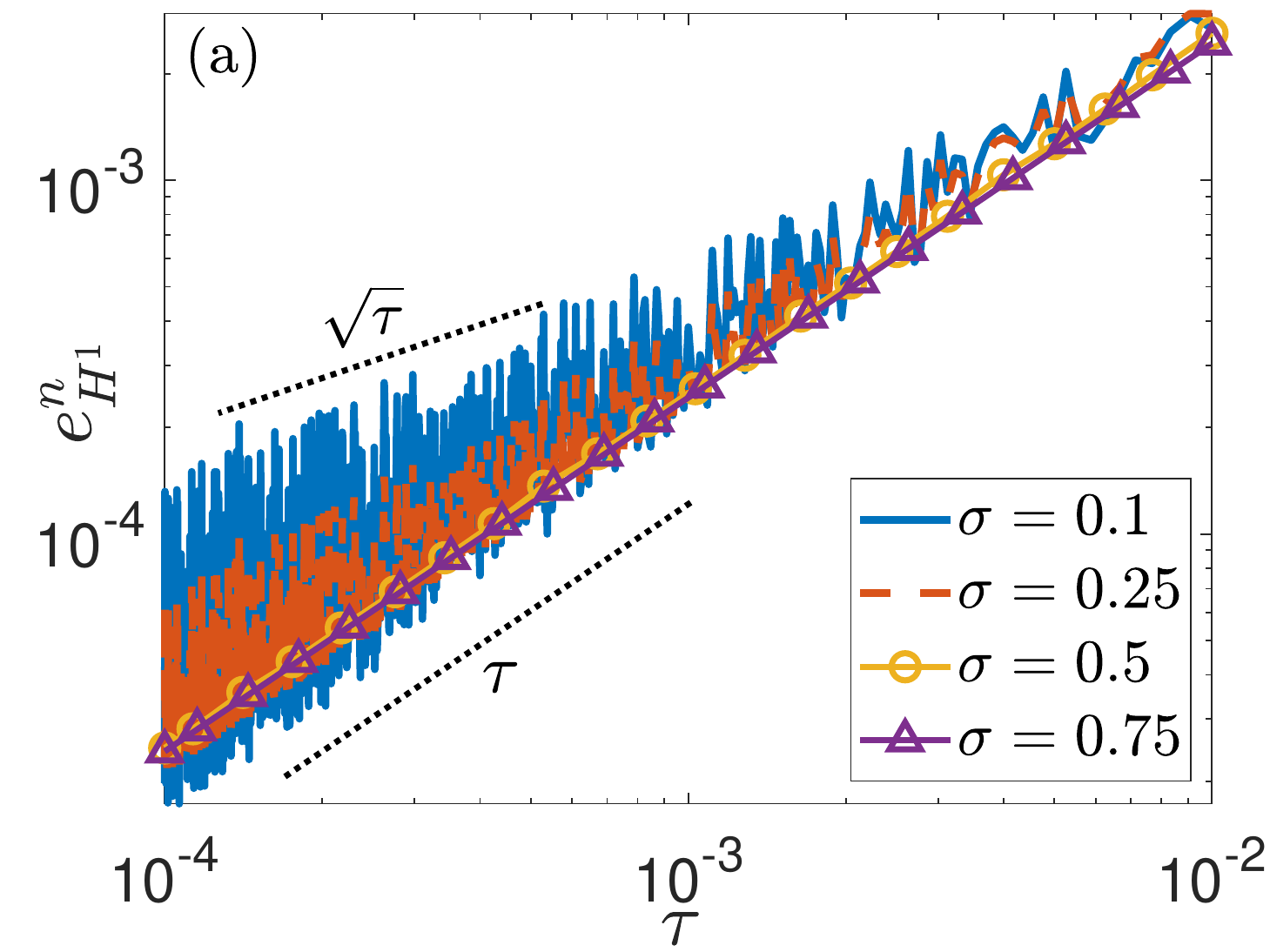}}
	{\includegraphics[width=0.475\textwidth]{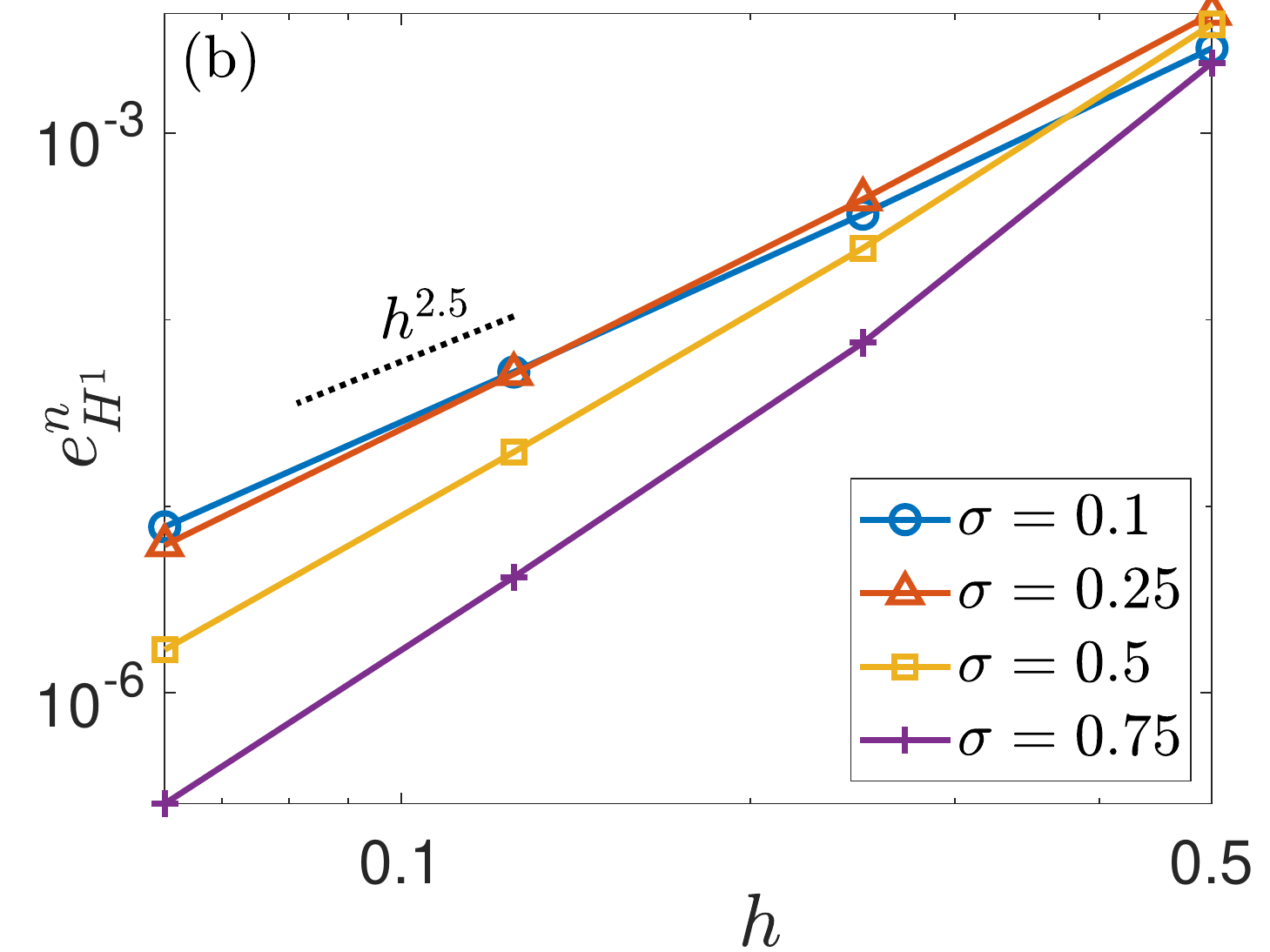}}
	\caption{Temporal errors (a) and spatial errors (b) in $ H^1 $-norm for $ \sigma = 0.1, 0.25, 0.5, 0.75 $ with Type I initial datum \cref{typeI_ini}. }
	\label{fig:exmp3_H1_multi_sigma}
	\end{figure}

\section{Conclusion}
Error bounds of the Lie-Trotter splitting sine pseudospectral method for the nonlinear Schr\"odinger equation (NLSE) with semi-smooth nonlinearity $ f(\rho) = \rho^\sigma (\sigma > 0)$ were established. For $0<\sigma\leq\frac{1}{2}$, we prove error bounds at $ O(\tau^{\frac{1}{2}+\sigma} + h^{1+2\sigma}) $ in $ L^2 $-norm without any CFL-type time step size restrictions, where $\tau>0$ and $h>0$ are the time step size and mesh size respectively. For $\sigma\geq\frac{1}{2}$, error bounds at $ O(\tau + h^{2}) $ in $ L^2 $-norm and at $ O(\tau^\frac{1}{2} + h) $ in $ H^1 $-norm are proved with mild time step size restrictions. In addition, when $\frac{1}{2}<\sigma<1$ and under the assumption of $ H^3 $-solution of the NLSE, we show an error bound at $ O(\tau^{\sigma} + h^{2\sigma}) $ in $ H^1 $-norm. Numerical results are reported to demonstrate our error estimates.


\section*{Appendix}
\begin{proof}[{Proof of \cref{discrete_embedding}}]
	We shall present the proof in 1D, and one can easily generalize it to higher dimensions. By triangle inequality, recalling that $\phi_j = \phi(x_j)$ for $j \in \mathcal{T}_N^0$, 
	\begin{align}\label{dGN-decomp}
		\| \dxp \phi \|_{l^4}^4
		&= h\sum_{j=0}^{N-1} |\dxp \phi_{j}|^4 = h\sum_{j=0}^{N-1} |(\dxp \phi_{j})^2 - (\dxp \phi_{0})^2 + (\dxp \phi_{0})^2||\dxp \phi_{j}|^2 \notag \\
		&\leq h\sum_{j=0}^{N-1} |(\dxp \phi_{j})^2 - (\dxp \phi_{0})^2||\dxp \phi_{j}|^2 + h\sum_{j=0}^{N-1} |\dxp \phi_{0}|^2|\dxp \phi_{j}|^2 \notag \\
		&= h\sum_{j=0}^{N-1} \left |\sum_{l=0}^{j-1}(\dxp \phi_{l+1})^2 - (\dxp \phi_{l})^2 \right ||\dxp \phi_{j}|^2 + |\dxp \phi_{0}|^2 \|\dxp \phi \|_{l^2}^2 =: K_1 + K_2.
	\end{align}
	We have separated the boundary terms $K_2$ from $K_1$. We start with the estimate of $K_1$, which is standard and can be obtained from the proof of the second inequality of (3.3) in \cite{bao2013}. We show it here for the convenience of the reader. Define the central difference operator $\dxx$ as
	\begin{equation}\label{dxx_def}
		\dxx v_{j} := \frac{v_{j+1} - 2 v_j + v_{j-1}}{h^2} = \frac{\dxp v_j - \dxp v_{j-1}}{h}, \quad 1 \leq j \leq N-1, \quad v \in Y_N. 
	\end{equation}
	For $K_1$ defined in \cref{dGN-decomp}, using triangle inequality and Cauchy inequality, we get
	\begin{align}\label{K_1_est}
		K_1
		&\leq h\sum_{j=0}^{N-1} \sum_{l=0}^{N-2} |\dxp \phi_{l+1} + \dxp \phi_{l} ||\dxp \phi_{l+1} - \dxp \phi_{l} ||\dxp \phi_{j}|^2 \notag \\
		&\leq h\sum_{j=0}^{N-1} |\dxp \phi_{j}|^2 \sqrt{\sum_{l=0}^{N-2} |\dxp \phi_{l+1} + \dxp \phi_{l} |^2} \sqrt{\sum_{l=0}^{N-2} |\dxp \phi_{l+1} - \dxp \phi_{l} |^2}  \notag \\
		&\lesssim \| \dxp \phi \|_{l^2}^2 \sqrt{h\sum_{l=0}^{N-1} |\dxp \phi_{l} |^2} \sqrt{h\sum_{l=1}^{N-1} |\dxx \phi_{l} |^2} =\| \dxp \phi \|_{l^2}^3 \sqrt{h\sum_{l=1}^{N-1} |\dxx \phi_{l} |^2}.
	\end{align}
	Since $ \phi \in X_N $, we have
	\begin{equation}\label{v_expansion}
		\phi_{j} = \phi(x_j) = \sum_{l=1}^{N-1} \widehat \phi_{l} \sin(\mu_l (x_j - a)) = \sum_{l=1}^{N-1} \widehat \phi_{l} \sin(\mu_l j h), \quad 0 \leq j \leq N, 
	\end{equation}
	which implies, by recalling \cref{dxx_def}, 
	\begin{equation}\label{eq:dxx_v}
		\dxx \phi_j = \frac{\phi_{j+1} - 2 \phi_j + \phi_{j-1}}{h^2} = -\sum_{l=1}^{N-1} \mu_l^2 \widehat \phi_{l} \left |\sinc(\mu_l h/2) \right |^2 \sin(\mu_l j h), 
	\end{equation}
	where $\sinc(x)=\sin(x)/x$ for $x \in \R$ with $\sinc(0) = 1$. By Parseval's identity, noting \cref{eq:dxx_v} and $|\sinc(x)| \leq 1$ for $x \in \R$, we get (similar to the proof of \cref{eq:H1_I_N})
	\begin{equation}\label{K_1_est_2}
		h\sum_{j=1}^{N-1} |\dxx \phi_{j} |^2 = \frac{Nh}{2} \sum_{l=1}^{N-1} \mu_l^4 \left |\widehat \phi_{l} \right |^2 \left |\sinc(\mu_l h/2) \right |^4 \leq \frac{Nh}{2} \sum_{l=1}^{N-1} \mu_l^4 \left |\widehat \phi_{l} \right |^2 \leq \| \phi \|_{H^2}^2. 
	\end{equation}
	Plugging \cref{K_1_est_2} into \cref{K_1_est} and using \cref{eq:H1_I_N}, we have
	\begin{equation}\label{K_1}
		K_1 \lesssim \| \phi \|_{H^2} \| \phi \|_{H^1}^3 \leq \| \phi \|_{H^2}^4. 
	\end{equation}
	For $K_2$ in \cref{dGN-decomp}, recalling \cref{v_expansion} and $|\sinc(x)| \leq 1$ for $x \in \R$, and noting that $\sum_{l=1}^\infty |\mu_l|^{-2} < \infty$, we have, by Cauchy inequality, 
	\begin{equation*}
		| \dxp \phi_{0} |  = \left |\frac{\phi_1}{h} \right | =\left | \sum_{l=1}^{N-1} \mu_l \widehat \phi_{l} \sinc(\mu_l h) \right | \leq \sqrt{\sum_{l=1}^{N-1} |\mu_l|^4 |\widehat \phi_{l}|^2}  \sqrt{\sum_{l=1}^{N-1} |\mu_l|^{-2}} \lesssim \| \phi \|_{H^2},
	\end{equation*}
	which implies, by using \cref{eq:H1_I_N} again, 
	\begin{equation}\label{K_2}
		K_2 = |\dxp \phi_{0}|^2 \|\dxp \phi \|_{l^2}^2 \lesssim \| \phi \|_{H^2}^2 \| \nabla \phi \|_{H^1}^2 \leq \| \phi \|_{H^2}^4. 
	\end{equation} 
	Plugging \cref{K_1,K_2} into \cref{dGN-decomp} yields the desired result.
\end{proof}

\bibliographystyle{myamsplain}

\end{document}